\documentclass[12pt,reqno]{amsproc}

\usepackage{amsmath, mathrsfs, amsfonts, amsthm, amssymb}
\usepackage{mdframed}
\usepackage{amsthm}
\usepackage{rotating}
\usepackage{graphicx}
\usepackage{hyperref}
\usepackage{enumitem}

\usepackage[left]{lineno}

\hypersetup{
  colorlinks   =  true, 
  urlcolor     = black, 
  linkcolor    = black, 
  citecolor   = black 
}

 \theoremstyle{plain}

\def\A{\mathcal A}
\def\B{\mathcal B}
\def\F{\mathcal F}

\def\K{\mathcal K}

\def\B{\mathcal B}

\def\H{\mathcal H}
\def\K{\mathcal K}

\def\M{\mathcal M}
\def\N{\mathcal N}
\def\P{\mathcal P}
\def\R{\mathcal R}

\def\NNN{\mathbb N}

\newtheorem{theorem}{Theorem}[section]

\newtheorem{corollary}[theorem]{Corollary}

\newtheorem{definition}[theorem]{Definition}
\newtheorem{example}[theorem]{Example}

\newtheorem{lemma}[theorem]{Lemma}

\newtheorem{problem}[theorem]{Problem}

\newtheorem{remark}[theorem]{Remark}

\makeatletter

\newcommand{\LeftEqNo}{\let\veqno\@@leqno}

\makeatother

\makeatletter
\newcommand{\subalign}[1]{%
  \vcenter{%
    \Let@ \restore@math@cr \default@tag
    \baselineskip\fontdimen10 \scriptfont\tw@
    \advance\baselineskip\fontdimen12 \scriptfont\tw@
    \lineskip\thr@@\fontdimen8 \scriptfont\thr@@
    \lineskiplimit\lineskip
    \ialign{\hfil$\m@th\scriptstyle##$&$\m@th\scriptstyle{}##$\crcr
      #1\crcr
    }%
  }
}
\makeatother

\numberwithin{equation}{section}

\begin{document}

\

\vspace{-2cm}


\title[Normed ideal perturbation of irreducible operators]{Normed ideal perturbation of irreducible operators in semifinite von Neumann factors}

\author{Rui Shi}
\address{School of Mathematical Sciences, Dalian University of
Technology, Dalian, 116024, China}
\email{ruishi@dlut.edu.cn, ruishi.math@gmail.com}


\subjclass[2010]{Primary 47C15}


\keywords{factor von Neumann alegbra, irreducible operator, normed ideal}

\begin{abstract}
	In \cite{Hal}, Halmos  proved an interesting result that the set of irreducible operators is dense in $\B(\H)$ in the sense of Hilbert-Schmidt approximation. In a von Neumann algebra $\M$ with separable predual, an operator $a\in \M$ is said to be \emph{irreducible in} $\M$ if $W^*(a)$ is an irreducible subfactor of $\M$, i.e., $W^*(a)'\cap \M={\mathbb C} \cdot I$. In this paper, let $\Phi(\cdot)$ be a $\Vert\cdot\Vert$-dominating, unitarily invariant norm (see Definition \ref{prelim_def1}), where by $\Vert\cdot\Vert$ we denote the operator norm. We prove that in every semifinite von Neumann factor $\M$ with separable predual, if the norm $\Phi(\cdot)$ satisfies a natural restriction introduced in (\ref{not-equiv-to-trace-norm}), then irreducible operators are $\Phi(\cdot)$-norm dense in $\M$. In particular, the operator norm $\Vert\cdot\Vert$ and the $\max\{\Vert\cdot\Vert, \Vert\cdot\Vert_p\}$-norm (for each $p>1$) naturally satisfy the condition in (\ref{not-equiv-to-trace-norm}), where $\tau$ is a faithful, normal, semifinite, tracial weight and $\Vert x\Vert_p=\tau(|x|^p)^{1/p}$ for all $x\in \M \cap L^{p}(\M,\tau)$ (see \cite[Preliminaries]{Pisier}). This can be viewed as a (stronger) analogue of a theorem of Halmos in \cite{Hal}, proved with different techniques developed in semifinite, properly infinite von Neumann factors. 
	
	Meanwhile,  for every $\Vert\cdot\Vert$-dominating, unitarily invariant norm $\Phi(\cdot)$, we develop another method to prove that each normal operator in $\M$ is a sum of an irreducible operator in $\M$ and an arbitrarily small $\Phi(\cdot)$-norm perturbation, where the $\Phi(\cdot)$-norm isn't restricted by (\ref{not-equiv-to-trace-norm}).   Particularly, the $\Phi(\cdot)$-norm can be the $\max\{\Vert\cdot\Vert, \Vert\cdot\Vert_1\}$-norm.    
\end{abstract}

\maketitle

\section{Introduction}

Let $\H$ be a Hilbert space.  We denote by $\mathcal{B}(\mathcal{H})$ the set of bounded linear operators on $\H$. Recall that an operator $a\in \B(\H)$ is \emph{irreducible} if $a$ has no nontrivial reducing subspaces. That is, if $p$ is a projection in $\mathcal{B}(\mathcal{H})$ such that $pa=ap$ then $p=0$ or $p=I$. 

A \emph{von Neumann  algebra} is a unital $\ast$-subalgebra of $\mathcal{B}(\mathcal{H})$ that is closed in the weak operator topology and contains the identity $I$. A \emph{factor} (or  \emph{von Neumann factor}) is a von Neumann algebra whose center consists of scalar multiples of the identity. Factors are   classified by  Murray and von Neumann into type I$_n$,  I$_\infty$,  II$_1$, II$_\infty$, and  III factors (see \cite {Murray}). By definition, $\mathcal B(\mathcal H)$ is a type I factor.

\emph{In the current paper, Hilbert spaces are always assumed to be complex and separable.} In \cite[Theorem]{Hal}, Halmos proved that the set of irreducible operators on $\H$ is a dense $G_\delta$ subset of $\B(\H)$ in the $\Vert\cdot\Vert$-norm topology, where by $\Vert\cdot\Vert$-norm we denote the operator norm. In \cite{Rad}, Radjavi and Rosenthal gave another short proof.  In addition, Halmos also mentioned that the set of irreducible operators is dense in $\B(\H)$ in the sense of Hilbert-Schmidt approximation at the end of \cite[Section 1]{Hal}. In the rest of the current paper, we will refer to this approximation result as \emph{Halmos' theorem}.

Inspired by Halmos' theorem, we can naturally extend the definition of irreducible operator in the setting of von Neumann factors.

\begin{definition}\label{irreducible-opt}
	Let $\M\subseteq \mathcal B(\mathcal H)$ be a   factor. An operator $a\in\M$ is said to be irreducible in $\M$, if $W^*(a)$ is an irreducible subfactor of $\M$, i.e., $W^*(a)'\cap \M={\mathbb C} I$, where $W^*(a)$ is the von Neumann algebra generated by $a$ and the identity $I$.
\end{definition}  

Actually, in each type of von Neumann factors with separable predual, there are irreducible operators. In \cite{Pearcy_2, Saito_2}, the authors proved, independently, that there exists a type ${\rm II}_{1}$ factor $\R$ with a single generator. Apparently, in $\R$, this generator is irreducible.  In \cite{Wogen}, Wogen showed that every properly infinite von Neumann algebra on a separable Hilbert space is singly generated. Thus, in a type ${\rm II}_{\infty}$ factor or a type ${\rm III}$ factor on $\H$, each single generator is irreducible. Recently, the authors of \cite{Fang} proved that in each type of von Neumann factor $\M$ with separable predual, the set of irreducible operators in $\M$ is $G_{\delta}$ and $\Vert \cdot \Vert$-norm dense. 

Recall that a factor $\M$ acting on $\H$ is semifinite and properly infinite, if $\M$ is either of type ${\rm I}_{\infty}$ or of type ${\rm II}_{\infty}$. In this case, we can further assume that there exists a faithful, normal, semifinite, tracial weight $\tau$ on $\M$. By virtue of Theorem 6.8.7 of \cite{Kadison2}, the $\Vert\cdot\Vert$-norm closure $\K(\M,\tau)$ of the two-sided ideal $\F(\M,\tau)=\{x\in\M: \tau(R(x))<\infty\}$ is the only proper, $\Vert\cdot\Vert$-norm closed, two sided ideal in the factor $\M$, where by $R(x)$ we denote the range projection of $x$, for every $x\in\M$.

In the current paper, first, we prove an analogue of Halmos' theorem in the setting of semifinite von Neumann factors with separable predual, with respect to a $\Vert\cdot\Vert$-dominating, unitarily invariant norm satisfying a natural restriction. Precisely, we prove the following theorem.

\vspace{0.2cm} {T{\scriptsize HEOREM}} \ref{main-thm}.
\emph{Suppose that $(\mathcal{M},\tau)$ is a semifinite, properly infinite von Neumann factor with separable predual, where $\tau$ is a faithful, normal, semifinite, tracial weight. Let $\mathcal{K}_{\Phi}(\mathcal{M},\tau)$ be a normed ideal of $(\mathcal{M},\tau)$ equipped with a $\Vert\cdot\Vert$-dominating, unitarily invariant norm $\Phi(\cdot)$ defined as in Definition $\ref{prelim_def1}$. Assume that
	\begin{equation}\label{not-equiv-to-trace-norm}
	\displaystyle
	\lim_{\substack{\tau(e)\rightarrow\infty\\ e\in \mathcal{P}\mathcal{F}(\mathcal{M},\tau)}}\frac{\Phi(e)}{\tau(e)} =0. 
	\end{equation}
	Then for each $x\in \mathcal{M} $ and every $\epsilon>0$, there exists an irreducible operator $y$ in $\mathcal{M}$ such that
	\begin{enumerate}
		\item[$(i)$] \quad $x-y\in{\mathcal{K}^{0}_{\Phi}(\mathcal{M},\tau)^{}}$ \quad  (see Definition $\ref{prelim_def1}$); 
		\item[$(ii)$] \quad $\Phi(x-y)\le\epsilon$. 
	\end{enumerate}
	In other words, if a $\Vert\cdot\Vert$-dominating, unitarily invariant norm $\Phi(\cdot)$ satisfies $(\ref{not-equiv-to-trace-norm})$, then irreducible operators in $\M$ are $\Phi(\cdot)$-norm dense in $(\M,\tau)$.}

We make several quick comments about Theorem \ref{main-thm}. (1) The result is also true without the assumption that $\M$ is `properly infinite'. When $\M$ is finite, the proof of Theorem \ref{main-thm}  is almost the same and easier than the case when $\M$ is properly infinite.  If $(\mathcal{M},\tau)$ is properly infinite, then $\mathcal{K}_{\Phi}(\mathcal{M},\tau)$ is a `nontrivial' normed ideal of $(\mathcal{M},\tau)$. This case is more interesting for Theorem  \ref{main-thm}. (2) Note that Halmos' proof in \cite{Hal} is based on minimal projections and there are no minimal projections in type ${\rm II}$ factors. Thus to prove Theorem \ref{main-thm}, we develop new techniques. (3) In the setting of $\B(\H)$, the restriction $(\ref{not-equiv-to-trace-norm})$ holds if and only if that the $\Phi(\cdot)$-norm is not equivalent to the $\Vert\cdot\Vert_1$-norm  (see \cite[Lemma 1]{Kuroda}). Enlightened by this characterization, we further prove the following theorem.

\vspace{0.2cm} {T{\scriptsize HEOREM}} \ref{normal-to-irreducible}.
\emph{Suppose that $(\mathcal{M},\tau)$ is a semifinite von Neumann factor with separable predual, where $\tau$ is a faithful, normal, semifinite, tracial weight. Let $\mathcal{K}_{\Phi}(\mathcal{M},\tau)$ be a normed ideal of $(\mathcal{M},\tau)$ equipped with a $\Vert\cdot\Vert$-dominating, unitarily invariant norm $\Phi(\cdot)$ $(\mbox{see Definition } \ref{prelim_def1})$.  }

	\emph{For each normal operator $x$ in $ \M$ and every $\epsilon>0$, there is an irreducible operator $y$ in $\mathcal{M}$ such that
	\begin{equation*}
		x-y\in{\mathcal{K}^{0}_{\Phi}(\mathcal{M},\tau)^{}} \quad \mbox{ and } \quad \Phi(x-y)\le\epsilon.
	\end{equation*}
	}

	Note that, in the setting of $\B(\H)$, the Kato-Rosenblum's theorem states that a self-adjoint operator $a\in\B(\H)$ with a non-vanishing spectrally absolutely continuous part, can't be expressed as a diagonal operator plus an arbitrarily small trace norm perturbation. This leads to the following question.
	
	\begin{problem}\label{Problem-1.2}
		Let $(\mathcal{M},\tau)$ be a semifinite, properly infinite von Neumann factor  with separable predual, where $\tau$ is a faithful, normal, semifinite, tracial weight. For every $\Vert\cdot\Vert$-dominating, unitarily invariant norm $\Phi(\cdot)$, are irreducible operators $\Phi(\cdot)$-norm dense in $\M$?
	\end{problem}

	By virtue of Theorem \ref{normal-to-irreducible}, every normal operator in $\M$ can be expressed as an irreducible operator in $\M$ up to an arbitrarily small  $\Phi(\cdot)$-norm perturbation. Furthermore, in terms of Lemma \ref{lemma-case-1} and Lemma \ref{lemma-case-2}, some non-normal operators in $\M$ can be also expressed as irreducible operators in $\M$ up to arbitrarily small  $\Phi(\cdot)$-norm perturbations.  These are positive evidences for Problem \ref{Problem-1.2}.
	
	The paper is organized as follows. In Section 2, we prepare necessary notation and results. A generalized Weyl-von Neumann type theorem for self-adjoint operators, normed ideals equipped with $\Vert\cdot\Vert$-dominating, unitarily invariant norms, and the universal irrational rotation algebras are recalled in this section. Besides, we cite Popa's results \cite{Pop,Popa2} about the existence of an irreducible hyperfinite subfactor $\N$ of a type II factor $\M$ and the existence of a Cartan masa $\A$ of $\N$ which is also a masa in $\M$. In Section 3, we prove Theorem \ref{main-thm} in two cases: the $\B(\H)$-case and the type $\rm{II}_{\infty}$ factor case. As we mentioned above, since there are no minimal projections in type $\rm{II}_{\infty}$ factors, we use techniques from universal irreducible rotation algebras and Popa's results to prove Theorem \ref{main-thm}. In Section 4, we develop a series of lemmas to prove Theorem \ref{normal-to-irreducible} for normal operators in $\M$ to be irreducible operators up to arbitrarily small, $\Vert\cdot \Vert$-dominating, unitarily invariant $\Phi(\cdot)$-norm perturbations.

\section{Preliminaries}

\subsection{Normed ideal perturbations and an extended Weyl-von Neumann theorem for self-adjoint operators in properly infinite, semifinite von Neumann algebras}

\ \newline

In this section, we prepare some useful lemmas for the main results. Since a generalized Weyl-von Neumann theorem for self-adjoint operators (Theorem 3.2.2 of \cite{Li}) is applied in the proof of the main result in the current paper, we briefly recall some definitions and results.

In 1909, Weyl \cite{Weyl} proved that a self-adjoint operator  in $\mathcal{B}(\mathcal{H})$ is a compact perturbation of a diagonal operator. Later, in 1935, von Neumann \cite{Von2} improved the result by replacing a  ``compact operator'' with an ``arbitrarily small Hilbert-Schmidt operator''.  

Recall that an operator $d$ in $\mathcal{B}(\mathcal{H})$ is called \emph{diagonal} if there exist a family $\{e_n\}_{n=1}^\infty$ of orthogonal projections in $\mathcal{B}(\mathcal{H})$ and a family $\{\lambda_n\}_{n=1}^\infty$ of complex numbers such that $d=\sum_{n=1}^\infty\lambda_ne_n$.

In \cite{Kuroda}, Kuroda generalized the Weyl-von Neumann theorem for every single self-adjoint operator in $\mathcal{B}(\mathcal{H})$ with respect to a unitarily invariant norm which is not equivalent to the trace norm. More specifically, given $\epsilon>0$ and $\Phi(\cdot)$ a unitarily invariant norm \emph{not equivalent to the trace norm}, for every self-adjoint operator $a$ in $\mathcal{B}(\mathcal{H})$, there exists a diagonal self-adjoint operator $d$ in $\mathcal{B}(\mathcal{H})$ such that $a-d$ is compact and $\Phi(a-d)\le \epsilon$.

To answer a problem attributed to Halmos concerning Hilbert-Schmidt perturbations of normal operators, Voiculescu \cite{Voi} proved that $n$-tuples of commuting self-adjoint operators, for $n\geq 2$, are $\mathscr{C}_n$-perturbations of diagonal $n$-tuples of commuting self-adjoint operators, where by $\mathscr{C}_n$ we denote the Schatten $n$-class operators in $\B(\H)$. 

In \cite{Voi}, to prove that every normal operator is a sum of a diagonal operator and an arbitrarily small Hilbert-Schmidt perturbation, Voiculescu developed important techniques associated with normed ideals (in terms of unitarily invariant norms introduced by Schatten). The reader is referred to \cite{Schatten} and \cite{Gohberg} for details about normed ideals.

Recently, the authors of \cite{Li} extended the definition of normed ideals in countably decomposable, semifinite, properly infinite von Neumann algebras. The definition is cited as follows.

\begin{definition}[Definition 2.1.1 of \cite{Li}]\label{prelim_def1} 
	Suppose that $(\M,\tau)$ is a countably decomposable, semifinite, properly infinite von Neumann algebra with a faithful, normal, semifinite, tracial weight $\tau$.  
	
	A {normed ideal} $\mathcal{K}_\Phi(\mathcal{M},\tau)$ of $(\mathcal{M},\tau)$ is a two sided ideal of $(\M,\tau)$ equipped with a norm $\Phi : \mathcal{K}_\Phi(\mathcal{M},\tau)\rightarrow [0,\infty)$, which satisfies
	\begin{enumerate}
		\item[(i)] \quad $\displaystyle \Phi(uxv)=\Phi(x)$ for all $x\in \mathcal{K}_\Phi(\mathcal{M},\tau)$ and unitary elements $u,v$ in $\mathcal{M}$, i.e. the norm $\Phi(\cdot)$ is unitarily invariant;
		\item[(ii)] \quad there exists a $\lambda>0$ such that $\displaystyle \Phi(x) \ge\lambda\|x\|$ for all $x\in \mathcal{K}_\Phi(\mathcal{M},\tau)$, i.e. the $\Phi(\cdot)$-norm is $\|\cdot \|$-dominating, where $\Vert\cdot \Vert$-norm means the operator norm;
		\item[(iii)] \quad $\mathcal{K}_\Phi(\mathcal{M},\tau)$ is a Banach space with respect to the norm $\Phi(\cdot)$;
		\item[(iv)] \quad $\displaystyle \mathcal{F}(\mathcal{M},\tau)=\{x\in\M: \tau(R(x))<\infty\}\subseteq \mathcal{K}_\Phi(\mathcal{M},\tau) \subseteq \mathcal{K }(\mathcal{M},\tau)$, where $\mathcal{K }(\mathcal{M},\tau)$ is the $\Vert\cdot \Vert$-norm closure of $\mathcal{F}(\mathcal{M},\tau)$ and $R(x)$ is the range projection of $x$, for every $x$ in $\M$.
	\end{enumerate}
	
	The $\Phi(\cdot)$-norm closure of $\mathcal{F}(\mathcal{M},\tau)$ in $\mathcal{K}_\Phi(\mathcal{M},\tau)$ will be denoted by $\mathcal{K}_\Phi^0(\mathcal{M},\tau)$, which is also a normed ideal of $(\mathcal{M},\tau)$. If $\mathcal{K}_\Phi^0(\mathcal{M},\tau)=\mathcal{K}_\Phi(\mathcal{M},\tau)$, then $\mathcal{K}_\Phi(\mathcal{M},\tau)$ is called a {minimal} normed ideal of $(\mathcal{M},\tau)$.
\end{definition}

\begin{remark}
	For the purpose of convenience, if $x\notin \mathcal{K}_\Phi(\mathcal{M},\tau)$, then we set $\Phi(x)=\infty$.
\end{remark}

\begin{example}
	The norm $\Phi(\cdot)$ is a natural generalization of the Schatten $p$-norm for $p\geq 1$. See \cite{Voi}, \cite{Schatten} and \cite{Gohberg} for examples of normed ideals when $(\mathcal M,\tau)=(\mathcal B(\mathcal H),\operatorname{Tr})$, where   $\mathcal H$ is a separable complex Hilbert space and $\operatorname{Tr}$ is the canonical trace of $\mathcal B(\mathcal H)$.
\end{example}

\begin{example}
	The $\Vert\cdot\Vert$-norm closure $\K(\M,\tau)$ of the two-sided ideal $\F(\M,\tau)$ is a normed ideal of $(\mathcal{M},\tau)$ with respect to the $\|\cdot\|$-norm.
\end{example}

For convenience, some useful properties of a normed ideal are listed in the next lemma.

\begin{lemma}[Lemma 2.1.5 of \cite{Li}]\label{prelim_lemma1} 
	Suppose that $\mathcal{K}_\Phi(\mathcal{M},\tau)$ is a normed ideal in $(\mathcal{M},\tau)$. Then the following statements are true.
	\begin{enumerate}
	\item[(i)] \quad $\Phi(axb)\le \|a\|\Phi(x) \|b\|, $ for all $x\in \mathcal{K}_\Phi(\mathcal{M},\tau)$ and $a,b\in\mathcal{M}$.
	
	\item[(ii)] \quad  If $x\in \mathcal{K}_\Phi(\mathcal{M},\tau)$, then $x^*\in\mathcal{K}_\Phi(\mathcal{M},\tau)$ and $\Phi(x)=\Phi(x^*)=\Phi(|x|)$.
	
	\item[(iii)] \quad  If $x, y\in \mathcal{K}_\Phi(\mathcal{M},\tau)$ with $0\le x\le y$, then $\Phi(x)\le \Phi(y)$.
	
	\item[(iv)] \quad  If $x \in \mathcal{F}(\mathcal{M},\tau)$, then $\Phi(x) \le \|x\|\Phi(R(x))$.
	
	\item[(v)] \quad  Suppose that $\{x_n\}_{n=1}^\infty\subseteq \mathcal{K}_\Phi^0(\mathcal{M},\tau)$ such that
	\begin{enumerate}
	\item [(1)] \quad $\sum_n x_n$ converges to $x\in\mathcal{M}$  in the weak$^{\ast}$-topology, and
	\item [(2)] \quad $\sum_n\Phi(x_n)<\infty$.
	\end{enumerate}
	Then $x\in \mathcal{K}_\Phi^0(\mathcal{M},\tau)$ and $\lim_k\Phi(x-\sum_{n=1}^k x_n)=0$.
	\end{enumerate}
\end{lemma}

The reader is referred to Lemma 2.1.5 of \cite{Li} for a quick proof. There are more examples of normed ideals of $(\mathcal{M},\tau)$ from the next lemma. Recall that $L^r(\mathcal{M},\tau)$, for $1\le r<\infty$, is the non-commutative $L^r$-space associated with $(\mathcal{M},\tau)$ and its norm $\|\cdot\|_r$ is defined by
\begin{equation}\label{p-norm}
\|x\|_r=(\tau(|x|^r)^{1/r}, \ \ \forall \ x\in L^r(\mathcal{M},\tau)
\end{equation}
(see \cite{Pisier} for more details).

\begin{lemma}[Lemma 2.1.6 of \cite{Li}]\label{prelim_lemma2} 
	Let $1\le r<\infty$ and $\mathcal{J}=L^r(\mathcal{M},\tau)\cap \mathcal{M}$. Define a mapping $\Phi(\cdot)$ on $\mathcal{J}$ by
	\begin{equation*}
		\Phi(x) =\max \{ \|x\|_r,\|x\| \}, \ \ \forall \ x\in \mathcal{J}.
	\end{equation*}
	Then $\mathcal{J}$ is a normed ideal of $(\mathcal{M},\tau)$ with respect to the norm $\Phi(\cdot)$. Furthermore, $\mathcal{J}$ is actually a minimal normed ideal of $(\mathcal{M},\tau )$ with respect to the norm $\Phi(\cdot)$.
\end{lemma}

\begin{definition}\label{prelim_def2} Let $1\le r<\infty$.  Define $\mathcal{K}_r(\mathcal{M},\tau)$ to be $L^r(\mathcal{M},\tau)\cap \mathcal{M}$ equipped with the norm $\Phi(\cdot)$ satisfying $\Phi(x) =\max \{ \|x\|_r,\|x\| \},$ for all $x\in L^r(\mathcal{M},\tau)\cap \mathcal{M}. $ Thus $\mathcal{K}_r(\mathcal{M},\tau)$ is a minimal normed ideal of $(\mathcal{M},\tau)$.
\end{definition}

In \cite{Li},  the authors proved a generalization of Kuroda's Theorem (Theorem of \cite{Kuroda}) in the setting of semifinite, properly infinite von Neumann algebras as follows.

\begin{definition}\label{diagonal-opt}
	Let $\M$ be a von Neumann algebra. An operator $d$ in $\mathcal{M}$ is  said to be diagonal if there exist a family $\{\lambda_n\}_{n=1}^\infty$ of complex numbers and a family $\{e_n\}_{n=1}^\infty$ of orthogonal projections in $\mathcal{M}$ satisfying $\sum^{\infty}_{n=1} e_n=I$ such that $d=\sum_{n=1}^\infty\lambda_ne_n$.
\end{definition}

\begin{theorem}[Theorem 3.2.2 of \cite{Li}]\label{Li-Shen-Shi-2}
	Let $\mathcal{M}$ be a countably decomposable, properly infinite von Neumann algebra with a faithful,  normal,  semifinite, tracial weight $\tau$ and let $\mathcal{K}_{\Phi}(\mathcal{M},\tau)$ be a normed ideal of $(\mathcal{M},\tau)$ {\rm (}see Definition $\ref{prelim_def1}${\rm)}. Assume that
	\begin{equation}
	\displaystyle
	\lim_{\substack{\tau(e)\rightarrow\infty\\ e\in \mathcal{P}\mathcal{F}(\mathcal{M},\tau)}}\frac{\Phi(e)}{\tau(e)} =0, \tag{\ref{not-equiv-to-trace-norm}}
	\end{equation}
	where $\P\F(\M,\tau)=\{e\in\F(\M,\tau):e=e^2=e^{\ast}\}$.
	Let $a\in \mathcal{M} $ be a self-adjoint element. Then for every $\epsilon>0$, there exists a diagonal operator $d$ in $\mathcal{M}$ such that
	\begin{enumerate}
		\item[(i)]\quad  $a-d\in{\mathcal{K}^{0}_{\Phi}(\mathcal{M},\tau)^{}}$; 
		\item[(ii)]\quad  $\Phi(a-d)\le\epsilon$. 
	\end{enumerate}
\end{theorem}

\begin{remark}\label{Phi-norm-control}
	We make two comments about Theorem $\ref{Li-Shen-Shi-2}$.  
	\begin{enumerate}
		\item \quad The reader is referred to Lemma $3.1.1$ of \cite{Li} for a characterization related to $(\ref{not-equiv-to-trace-norm})$. Specially, in the case of $\B(\H)$, Lemma $1$ of \cite{Kuroda} implies that a unitarily invariant norm $\Phi(\cdot)$ satisfies $(\ref{not-equiv-to-trace-norm})$ if and only if $\Phi(\cdot)$ is \textbf{not} equivalent to the trace norm.
		\item \quad Actually, in Theorem $\ref{Li-Shen-Shi-2}$, the diagonal operator $d$ is in the form
	\begin{equation}\label{diag-opt}
		d=\sum\nolimits^{\infty}_{n=1}\lambda^{\prime}_{n}e^{\prime}_{n}
	\end{equation}
	where $\{\lambda^{\prime}_{n}\}_{n\geq 1}\subseteq[-\Vert a\Vert,\Vert a\Vert]$ and $\{e^{\prime}_{n}\}_{n\geq 1}$ is a sequence of pairwise mutually orthogonal projections in $\F(\M,\tau)$ with $\sum_{n\geq 1}e^{\prime}_{n}=I$.
	
	Note that, by $(\ref{not-equiv-to-trace-norm})$, there exists an integer $m_0\geq 1$ such that
	\begin{equation*}
		\frac{\Phi(e)}{\tau(e)}\leq 1 \quad \mbox{ for }\ e\in \mathcal{P}\mathcal{F}(\mathcal{M},\tau) \ \mbox{ and }\ \tau(e)\geq m_0.
	\end{equation*}
	Thus for $e\in \mathcal{P}\mathcal{F}(\mathcal{M},\tau)$ and $m_0 \leq \tau(e)\leq m_0+1$, we have $\Phi(e)\leq \tau(e)\leq m_0+1$. It follows from $(iii)$ of Lemma $\ref{prelim_lemma1}$ that the inequality $\Phi(f)\leq \Phi(e)\leq m_0+1$ holds for each subprojection $f\leq e$ with $\tau(f)\leq m_0$. 
	
	If $\M$ is a countably decomposable, properly infinite type ${\rm II}_{\infty}$ factor with a faithful,  normal,  semifinite, tracial weight $\tau$, then $(i)$ of Definition $\ref{prelim_def1}$ and the preceding arguments guarantee that
	\begin{equation*}
		\tau(f)\leq m_0 \quad \Rightarrow \quad \Phi(f) \leq m_0+1. 
	\end{equation*}
	In terms of the diagonal form in $(\ref{diag-opt})$, there exists a partition $\{e_{n}\}_{n\geq 1}$ of the identity $I$ finer than $\{e^{\prime}_{n}\}_{n\geq 1}$ with each $e_{n}$ in $\mathcal{P}\mathcal{F}(\mathcal{M},\tau)$ and $\tau(e_{n})<m_0$ such that 
	\begin{equation}\label{diag-opt-finer}
		d=\sum\nolimits^{\infty}_{n=1}\lambda_{n}e_{n} \quad \mbox{ and } \quad \Phi(e_{n}) \leq m_0+1 \ \mbox{ for } \ n\geq 1, 
	\end{equation}
	where $\{\lambda_{n}\}_{n\geq 1}\subseteq[-\Vert a\Vert,\Vert a\Vert]$.
	\end{enumerate}
\end{remark}

By Remark \ref{Phi-norm-control}, we reformulate Theorem \ref{Li-Shen-Shi-2} in the setting of countably decomposable, properly infinite, semifinite von Neumann algebras as follows.

\begin{theorem}\label{Li-Shen-Shi-3}
	Let $\mathcal{M}$ be a countably decomposable, properly infinite von Neumann algebra with a faithful,  normal,  semifinite, tracial weight $\tau$ and let $\mathcal{K}_{\Phi}(\mathcal{M},\tau)$ be a normed ideal of $(\mathcal{M},\tau)$ {\rm (}see Definition $\ref{prelim_def1}${\rm)}. Assume that
	\begin{equation*}
	\displaystyle
	\lim_{\substack{\tau(e)\rightarrow\infty\\ e\in \mathcal{P}\mathcal{F}(\mathcal{M},\tau)}}\frac{\Phi(e)}{\tau(e)} =0. 
	\end{equation*}
	Let $a\in \mathcal{M} $ be a self-adjoint element. Then there exists an integer $m_0\geq 1$, and for every $\epsilon>0$, there exists a diagonal operator $d=\sum\nolimits^{\infty}_{n=1}\lambda_{n}e_{n}\in\mathcal{M}$ as in the form of $(\ref{diag-opt-finer})$ such that
	\begin{enumerate}
		\item[(i)]\quad  $ a-d\in{\mathcal{K}^{0}_{\Phi}(\mathcal{M},\tau)}$ \ and \ $\Phi(e_{n}) \leq m_0+1$ for all $n\geq 1$; 
		\item[(ii)]\quad  $\Phi(a-d)\le\epsilon$ \ and \ $\Vert d \Vert \leq \Vert a \Vert$. 
	\end{enumerate}
\end{theorem}

\subsection{Irrational rotation algebras $\A_{\theta}$ and hyperfinite type ${\rm II}_{1}$ factors}

\ \newline

The class of irrational rotation algebras $\mathcal{A}_{\theta}$ have been studied a lot in recent years. Let $\theta$ be an irrational number, the irrational rotation algebra $\mathcal{A}_{\theta}$ is the universal ${\rm C}^{\ast}$-algebra generated by two unitary elements $u$ and $v$ satisfying
\begin{equation*}
	uv=e^{2\pi i\theta}vu.
\end{equation*}
Since $\mathcal{A}_{\theta}$ can also be viewed as a crossed product ${\rm C}^{\ast}$-algebra, an application of Theorem 1 of \cite{Rosenberg} entails that $\mathcal{A}_{\theta}$ is amenable. Moreover, Elliott and Evans \cite{Elliott} proved that $\mathcal{A}_{\theta}$ is a limit circle algebra.

It is well known that $\mathcal{A}_{\theta}$ is simple and there exists only one faithful tracial state $\tau$ on $\mathcal{A}_{\theta}$. The reader is referred to Chapter VI of \cite{Davidson} for more details. By virtue of the GNS construction, the tracial state $\tau$ induces a $\ast$-representation $\pi$ of $\mathcal{A}_{\theta}$ on $L^{2}(\mathcal{A}_{\theta},\tau)$. It is easy to verify that $\pi$ is a unital $\ast$-isomorphism. Thus, $\pi(\mathcal{A}_{\theta})$ is also amenable. By applying Corollary $2$ of \cite{Connes}, $\pi(\mathcal{A}_{\theta})$ is nuclear. It follows that $\pi(\mathcal{A}_{\theta})^{\prime\prime}$ is injective in terms of Theorem IV.$2.2.13$ and Theorem IV.$3.1.12$ of \cite{Blackadar}. Note that the tracial state $\tau$ induces a tracial vector state on $\pi(\mathcal{A}_{\theta})^{\prime\prime}$. The following facts are useful:
\begin{enumerate}
	\item\quad  $\tau(u^{n})=\tau(v^{n})=0$ for each non-zero integer $n$;
	\item\quad  the set $\{u^{m}v^{n}:m,n\in\mathbb{Z}\}$ forms an orthonormal basis of $L^{2}(\mathcal{A}_{\theta},\tau)$.
\end{enumerate}
By the above facts, it can be verified that $\pi(\mathcal{A}_{\theta})^{\prime\prime}$ is an injective type ${\rm II}_{1}$ factor. Thus, Theorem $6$ of \cite{Connes2} entails that $\pi(\mathcal{A}_{\theta})^{\prime\prime}$ is hyperfinite.

\begin{lemma}\label{irreducible-opt}
	Let $a\in\pi(\mathcal{A}_{\theta})^{\prime\prime}$ be a self-adjoint operator such that
	\begin{equation*}
		\{a\}^{\prime\prime}=\{\pi(u),\pi(u)^{\ast}\}^{\prime\prime}. 
	\end{equation*} 
	If $b$ is a non-scalar self-adjoint operator in $\{\pi(v),\pi(v)^{\ast}\}^{\prime\prime}$, then $a+ib$ is an irreducible operator in the von Neumann algebra $\pi(\mathcal{A}_{\theta})^{\prime\prime}$.
\end{lemma}

\begin{proof}
	Note that $\{\pi(u),\pi(u)^{\ast}\}^{\prime\prime}$ and $\{\pi(v),\pi(v)^{\ast}\}^{\prime\prime}$ are both masas (short for maximal abelian self-adjoint algebra) in $\pi(\mathcal{A}_{\theta})^{\prime\prime}$. Let $p$ be a projection in $\pi(\mathcal{A}_{\theta})^{\prime\prime}$ such that $pa=ap$ and $pb=bp$. Since $a$ generates $\{\pi(u),\pi(u)^{\ast}\}^{\prime\prime}$, we have that $p$ belongs to $\{\pi(u),\pi(u)^{\ast}\}^{\prime\prime}$. Note that $p$ and $b$ can be also viewed as vectors in $L^{2}(\mathcal{A}_{\theta},\tau)$. It follows that $p$ can be expressed as a Laurent series of $u$ and $b$ can be expressed as a Laurent series of $v$. That $pb=bp$ guarantees that $p$ is trivial, since $\{u^{m}v^{n}: m, n\in\mathbb{Z}\}$ is an orthonormal basis of $L^{2}(\mathcal{A}_{\theta},\tau)$. Thus, $a+ib$ is irreducible in $\pi(\mathcal{A}_{\theta})^{\prime\prime}$.
\end{proof}

The following lemma is a special case of Corollary 4.1 of \cite{Pop} proved by Popa, which is useful in the proof of our main result.

\begin{lemma}\label{Popa}
Every type ${\rm II}_1$ factor $\M$ with separable predual contains an irreducible, hyperfinite subfactor $\N$, i.e., $\N^{\prime}\cap\M=\mathbb{C}$. Furthermore, there is a maximal abelian $\ast$-subalgebra of $\M$ wihch is regular in $\N$. 
\end{lemma}

Note that in the hyperfinite type ${\rm II}_{1}$ factor $\pi(\mathcal{A}_{\theta})^{\prime\prime}$, $\{\pi(u),\pi(u)^{\ast}\}^{\prime\prime}$ and $\{\pi(v),\pi(v)^{\ast}\}^{\prime\prime}$ are both Cartan masas. The following result is proved by Connes, Feldman, Weiss \cite{Connes2}, and Popa \cite{Popa2} separately.

\begin{lemma}[Theorem 4.1 of \cite{Popa2}]\label{Popa2}
	If $\A_1$ and $\A_2$ are Cartan subalgebras of the hyperfinite type ${\rm II}_1$ factor $\R$, then there exists a normal  $\ast$-automorphism $\pi\in \operatorname{Aut}(\R)$ such that $\pi(\A_1)=\A_2$.
\end{lemma}

\section{Main results}

In this section, we will prove an extended Halmos' theorem in semifinite factors with separable predual. For this purpose, we deal with factors of type ${\rm I}_{\infty}$  and type ${\rm II}_{\infty}$ separately. This is because type ${\rm II}_{\infty}$ factors contain no minimal projections while minimal projections play an important role in the study of factors of type ${\rm I}_{\infty}$. It follows that the proofs in these two cases are different in details. Recall that we always assume that $\mathcal{H}$ is a complex separable Hilbert space and we let $\B(\H)$ denote the set of bounded linear operators on $\mathcal{H}$.  

Let $(\M,\tau)$ be a semifinite von Neumann factor with a faithful, normal, semifinite, tracial weight $\tau$. Recall that $\displaystyle \mathcal{F}(\mathcal{M},\tau)=\{x\in\M: \tau(R(x))<\infty\}$. Every operator in $\mathcal{F}(\mathcal{M},\tau)$ is said to be of $(\mathcal{M},\tau)$-\emph{finite-rank} in this paper. When no confusion can arise, we just call $x$ a \emph{finite-rank} operator in $(\mathcal{M},\tau)$, for every $x$ in $\mathcal{F}(\mathcal{M},\tau)$. This coincides with the definition of finite rank operators in the setting of $\B(\H)$.

\subsection{Case 1: $\Phi(\cdot)$-norm-density of irreducible operators in Type ${\rm I}_{\infty}$ factors}

\ \newline

In $\B(\H)$, with Kuroda's theorem in \cite{Kuroda}, we can extend Halmos' theorem with respect to each unitary invariant norm \textbf{not} equivalent to the trace norm. For completeness, we sketch its proof in this subsection.

\begin{lemma}\label{compact-perturbation}
	Let $(\M,\tau)$ be a countably decomposable, semifinite, properly infinite von Neumann factor with a faithful, normal, semifinite, tracial weight $\tau$. Let $\mathcal{K}_{\Phi}(\mathcal{M},\tau)$ be a normed ideal of $(\mathcal{M},\tau)$ equipped with a norm $\Phi(\cdot)$ defined as in Definition $\ref{prelim_def1}$. 
	
	For $\epsilon>0$, let $\{\alpha_{i}\}_{i\geq 1}$ be a subsequence of $\{\epsilon/2^{k}\}_{k\geq 1}$. If $\{e_{i}\}_{i\geq 1}$ is a sequence of mutually orthogonal, finite-rank projections in $\M$ such that
	\begin{enumerate}
		\item \quad $\sum_{i\geq 1}e_{i}=I$;
		\item \quad $\Phi(e_{i})\leq \lambda_{1}$ for some $\lambda_{1}>0$ and all $i\geq 1$,
	\end{enumerate}  
	then it follows that
	\begin{equation*}
		\sum\nolimits_{i\geq 1}{\alpha}_{i}e_i\in{\mathcal{K}^{0}_{\Phi}(\mathcal{M},\tau)^{}}\quad \mbox{and}\quad 
		\Phi(\sum\nolimits_{i\geq 1}{\alpha}_{i}e_i)\le\epsilon\cdot\lambda_{1}. 
	\end{equation*} 
\end{lemma}

\begin{proof}
	By Definition $\ref{prelim_def1}$, the self-adjoint operator $f_{n}:=\sum_{1\leq i \leq n}{\alpha}_{i}e_i\in\F(\M,\tau)$. The inequality
	\begin{equation*}
		\Phi(f_{n+m}-f_{n})=\Phi(\sum\nolimits_{1\leq i \leq m}{\alpha}_{n+i}e_{n+i})\leq \sum\nolimits_{1\leq i \leq m}{\alpha}_{n+i}\Phi(e_{n+i})<\frac{\epsilon\cdot\lambda_{1}}{2^{n}}. 
	\end{equation*}
	guarantees that $\{f_{n}\}_{n\geq 1}$ is a Cauchy sequence with respect to the norm $\Phi(\cdot)$ and $\sum\nolimits_{i\geq 1}{\alpha}_{i}e_i$ belongs to ${\mathcal{K}^{0}_{\Phi}(\mathcal{M},\tau)^{}}$. A routine calculation implies that $\Phi(\sum\nolimits_{i\geq 1}{\alpha}_{i}e_i)\le\epsilon\cdot\lambda_{1}$.
\end{proof}

\begin{remark}\label{rem-B(H)}
	Suppose that $\mathcal{M}=\B(\H)$ and $\tau$ is a faithful, normal, semifinite, tracial weight on $\B(\H)$. Note that, as characterized in Lemma $1$ of \cite{Kuroda},  a unitarily invariant norm $\Phi(\cdot)$ on $\F(\M,\tau)$ is not equivalent to the trace norm if and only if
	\begin{equation*}
	\displaystyle
	\lim_{\substack{\tau(e)\rightarrow\infty\\ e\in \mathcal{P}\mathcal{F}(\mathcal{M},\tau)}}\frac{\Phi(e)}{\tau(e)} =0. 
	\end{equation*}
	Enlightened by this characterization, we have the following theorem.
\end{remark}

\begin{theorem}\label{B(H)-case}
Let $\H$ be a separable, infinite-dimensional Hilbert space. Suppose that $\mathcal{M}=\B(\H)$ and $\tau$ is a faithful, normal, semifinite, tracial weight on $\B(\H)$. Let $\mathcal{K}_{\Phi}(\mathcal{M},\tau)$ be a normed ideal of $(\mathcal{M},\tau)$ equipped with a unitarily invariant norm $\Phi(\cdot)$ (defined in Definition $\ref{prelim_def1}$). Assume that
	\begin{equation*}
	\displaystyle
	\lim_{\substack{\tau(e)\rightarrow\infty\\ e\in \mathcal{P}\mathcal{F}(\mathcal{M},\tau)}}\frac{\Phi(e)}{\tau(e)} =0. 
	\end{equation*}
	Then for each $x\in \mathcal{M} $ and every $\epsilon>0$, there exists an irreducible operator $y$ in $\mathcal{M}$ such that
	\begin{enumerate}
		\item[(i)] \quad  $x-y\in{\mathcal{K}^{0}_{\Phi}(\mathcal{M},\tau)}$; 
		\item[(ii)]\quad  $\Phi(x-y)\le\epsilon$. 
	\end{enumerate}
\end{theorem}

\begin{proof}
	Let $x\in \M$ and $\epsilon>0$. Write $x=a+ib$, where $a$ and $b$ are self-adjoint operators in $\M$. By applying Theorem \ref{Li-Shen-Shi-3}, there is a diagonal, self-adjoint operator $a_{1}:=\sum_{i\geq 1}\alpha^{\prime}_{i}e_{i}$ in $\B(\H)$ with $\{\alpha^{\prime}_{i}\}_{i\geq 1}\subseteq[-\Vert a\Vert,\Vert a\Vert]$ such that
	\begin{enumerate}
		\item\quad  $e_{i}$ and $e_{j}$ are mutually orthogonal projections for $i\neq j$;
		\item\quad  each $e_{i}$ is a minimal projection in $\B(\H)$ with $\sum_{i\geq 1}e_{i}=I$;
		\item\quad  $a-a_{1}\in{\mathcal{K}^{0}_{\Phi}(\mathcal{M},\tau)^{}}$; 
		\item\quad  $\Phi(a-a_{1})\le\epsilon/4$.
	\end{enumerate}
	As an application of Lemma \ref{compact-perturbation}, there is a diagonal operator $a_{2}:=\sum_{i\geq 1}\alpha_{i}e_{i}$ in $\B(\H)$ with $\{\alpha_{i}\}_{i\geq 1}\subseteq[-\Vert a \Vert,\Vert a \Vert]$ such that
	\begin{enumerate}
		\item\quad  $\alpha_{i}\neq\alpha_{j}$ for $i\neq j$;
		\item\quad  $a_{1}-a_{2}\in{\mathcal{K}^{0}_{\Phi}(\mathcal{M},\tau)^{}}$; 
		\item\quad  $\Phi(a_{1}-a_{2})\le\epsilon/4$.
	\end{enumerate}
	Note that the construction of $a_{2}$ entails that each operator in $\B(\H)$ commuting with $a_{2}$ is diagonal with respect to $\{e_{i}\}_{i\in\NNN}$.

	Corresponding to $\{e_{i}\}_{i\geq 1}$, there is a system of matrix units $\{e_{ij}\}_{i,j\in \NNN}$ for $\B(\H)$ such that
	\begin{enumerate}
		\item\quad  $e_i=e_{ii}$ for all $i\in\NNN$;
		\item\quad  $e_{ij}=e^{\ast}_{ji}$, for all $i$ and $j$ in $\NNN$;
		\item\quad  $e_{mn}e_{ij}=\delta_{ni}e_{mj}$ for all  $m, n, i, j \in\NNN$;
		\item\quad  {\scriptsize {SOT}}-${\rm }\sum^{\infty}_{i=1} e_{ii}=I$,
	\end{enumerate} 
	where $\delta_{n i}$ is the Kronecker symbol.

	With respect to $\{e_{ij}\}_{i,j\in \NNN}$, the self-adjoint operator $b$ can be expressed as
	\begin{equation*}
		b:=\sum_{i,j\geq 1}\beta^{\prime}_{ij}e_{ij},
	\end{equation*} 
	where $\beta^{\prime}_{ij}$'s are complex numbers. Note that $\beta^{\prime}_{ij}\beta^{\prime}_{ji}=|\beta^{\prime}_{ij}|^{2}$. For the sake of simplicity, the entries $\beta^{\prime}_{ij}$'s satisfying $i+k=j$ are said to be in the \emph{$k$-diagonal} of $\sum_{i,j\geq 1}\beta^{\prime}_{ij}e_{ij}$. Now we focus on the the 1-diagonal. It is easy to check that the operator $v:=\sum_{i \ge 1}e_{i,i+1}$ is a partial isometry. 
	
	Construct a self-adjoint operator $b_2:=\sum_{i,j\ge 1}\beta_{ij}e_{ij}$ and a diagonal, self-adjoint operator $c:=\sum_{i\ge 1}\gamma_{i}e_{ii}$ as follows:
	\begin{equation*}
		\beta_{ij}:=\begin{cases}
		\beta^{\prime}_{ij},& \mbox{if } |i-j|\neq 1;\\
		\beta^{\prime}_{ij},& \mbox{if } |i-j|=1 \mbox{ and }\beta^{\prime}_{ij}\neq 0;\\
		\epsilon/2^{i+2},&\mbox{if } j=i+1 \mbox{ and }\beta^{\prime}_{ij}= 0;\\
		\epsilon/2^{j+2},&\mbox{if } i=j+1 \mbox{ and }\beta^{\prime}_{ij}= 0.
		\end{cases}\quad \mbox{ and }\quad 
		\gamma_{i}:=\begin{cases}
		0,& \mbox{if }    \beta^{\prime}_{i,i+1}\neq 0;\\
		\epsilon/2^{i+2},&\mbox{if } \beta^{\prime}_{i,i+1}= 0.
		\end{cases} 
	\end{equation*}
Apparently, $b_2$ is self-adjoint and $b_2=b+cv+v^{\ast}c$. By Lemma \ref{compact-perturbation}, we have that
\begin{equation*}
	 b-b_{2}\in{\mathcal{K}^{0}_{\Phi}(\mathcal{M},\tau)^{}} \quad \mbox{ and } \quad 
	 \Phi(b-b_{2})\le\epsilon/2.
\end{equation*}

Define an operator $y:=a_{2}+ib_{2}$. We claim that $y$ is irreducible in $\B(\H)$. Actually, if $p$ is a projection in $\B(\H)$ commuting with $y$, then $p$ commutes with both $a_{2}$ and $b_{2}$. The equality $pa_2=a_2p$ implies that $p$ is diagonal. Thus, each diagonal entry of $p$ must be either $e_{ii}$ or $0$. Since each 1-diagonal entry of $b_2$ is non-zero, the equality $pb_2=b_2p$ guarantees that $p$ is trivial.

Note that
\begin{equation*}
\begin{aligned}
	\Phi(x-y)&=\Phi(a+ib-a_1+a_1-a_2-ib_{2}) \\
	&\leq \Phi(a-a_1)+\Phi(a_1-a_2)+\Phi(b-b_{2})\\
	&\leq \epsilon.
\end{aligned} 
\end{equation*} 
It follows that
\begin{equation*}
	x-y\in{\mathcal{K}^{0}_{\Phi}(\mathcal{M},\tau)^{}}\quad \mbox{ and }\quad \Phi(x-y)\le\epsilon.
\end{equation*}
This completes the proof.
\end{proof}

\subsection{Case 2: $\Phi$-norm-density of irreducible operators in type ${\rm II}_{\infty}$ factors}

\ \newline

Since type ${\rm II}_{\infty}$ factors contain no minimal projections, the proof of Theorem \ref{B(H)-case} doesn't work directly for any type ${\rm II}_{\infty}$ factor. Fortunately, the irrational rotation algebra $\mathcal{A}_{\theta}$ enables us to develop new techniques to extend Theorem \ref{B(H)-case}.

\begin{theorem}\label{II_infty-case}
Suppose that $(\mathcal{M},\tau)$ is a type ${\rm II}_{\infty}$ factor with separable predual, where $\tau$ is a faithful, normal, semifinite, tracial weight. Let $\mathcal{K}_{\Phi}(\mathcal{M},\tau)$ be a normed ideal of $(\mathcal{M},\tau)$ equipped with a $\Vert\cdot\Vert$-dominating, unitarily invariant norm $\Phi(\cdot)$ defined as in Definition $\ref{prelim_def1}$. 

Assume that
	\begin{equation*}
	\displaystyle
	\lim_{\substack{\tau(e)\rightarrow\infty\\ e\in \mathcal{P}\mathcal{F}(\mathcal{M},\tau)}}\frac{\Phi(e)}{\tau(e)} =0.
	\end{equation*}
	Then for each $x\in \mathcal{M} $ and every $\epsilon>0$, there exists an irreducible operator $y$ in $\mathcal{M}$ such that
	\begin{equation*}
		x-y\in{\mathcal{K}^{0}_{\Phi}(\mathcal{M},\tau)^{}} \quad \mbox{ and } \quad \Phi(x-y)\le\epsilon.
	\end{equation*}
\end{theorem}

\begin{proof}
	Fix $\epsilon>0$ and an operator $x:=a+ib$ in $\M$ such that $a$ and $b$ are self-adjoint operators in $\M$. First, we make perturbations of $a$ and $b$ respectively. Then we construct an irreducible operator $y$ in $\M$ as required. 
	
	By virtue of Theorem \ref{Li-Shen-Shi-3}, there exists a diagonal, self-adjoint operator $a_{1}:=\sum_{i\geq 1}\alpha^{\prime}_{i}e^{\prime}_{i}$ in $\M$ with $\{\alpha^{\prime}_{i}\}_{i\geq 1}\subseteq[-\Vert a\Vert,\Vert a\Vert]$ such that
	\begin{enumerate}
		\item\quad  $e^{\prime}_{i}$ and $e^{\prime}_{j}$ are mutually orthogonal projections for $i\neq j$ and $\sum_{i\geq 1}e^{\prime}_{i}=I$;
		\item\quad  each projection $e^{\prime}_{i}$ is of finite-rank and satisfies $\Phi(e^{\prime}_{i})\leq m_0+1$ for a uniformly upper bound $m_0>1$;
		\item\quad  $a-a_{1}\in{\mathcal{K}^{0}_{\Phi}(\mathcal{M},\tau)^{}}$; 
		\item\quad  $\Phi(a-a_{1})\le\epsilon/8$.
	\end{enumerate}

	With respect to $\{e^{\prime}_{i}\}_{i\geq 1}$, the self-adjoint operator $b$ can be expressed in the form
	\begin{equation*}
		b=\sum_{i,j\ge 1}b^{\prime}_{ij} \quad \mbox{ and } \quad b^{\prime}_{ij}:=e^{\prime}_{i}be^{\prime}_{j} \quad \mbox{ for all } i,j\in\NNN,
	\end{equation*} 
	where each $b^{\prime}_{ij}$ can be viewed as an operator from $\mbox{ran }e^{\prime}_{j}$ to $\mbox{ran }e^{\prime}_{i}$. Note that each $b^{\prime}_{ii}$ is self-adjoint. When we consider $b^{\prime}_{ii}$ as a self-adjoint operator restricted on $\mbox{ran }e^{\prime}_{i}$ for all $i\in\NNN$, by virtue of the spectral theorem, there are real numbers $\{\beta^{\prime}_{i,j}\}_{1\le j \le n_i}\subseteq [-\Vert b^{\prime}_{ii}\Vert, \Vert b^{\prime}_{ii}\Vert]$ and a finite partition $\{e^{\prime}_{i,j}\}_{1\le j \le n_i}$ of $e^{\prime}_{i}$ consisting of finitely many spectral projections of $b^{\prime}_{ii}$  such that
	\begin{enumerate}
		\item\quad  each $b^{\prime\prime}_{ii}:=\sum_{1\leq j \leq n_i}\beta^{\prime}_{i,j}e^{\prime}_{i,j}$ is self-adjoint;
		\item\quad  $\Phi(b^{\prime}_{ii}-b^{\prime\prime}_{ii})\leq\Vert b^{\prime}_{ii}-b^{\prime\prime}_{ii}\Vert\Phi(e^{\prime}_{i})<\epsilon/2^{i+3}$.
	\end{enumerate}
	For convenience, we rename $\{e^{\prime}_{i,j}\}_{1\le j \le n_i; 1\le i}$ as $\{e_{i}\}_{i\geq 1}$. Note that $\{e_{i}\}_{i\geq 1}$ is a finer partition of the identity operator $I$ relative to $\{e^{\prime}_{i}\}_{i\geq 1}$.

	Based on the matrix form $b:=\sum_{i,j\ge 1}b^{\prime}_{ij}$, we define an operator $b_{2}:=\sum_{i,j\ge 1}b_{ij}$ in the form:
	\begin{equation*}
		b_{ij}=\begin{cases}b^{\prime\prime}_{ii},& \mbox{for }i=j;\\
		b^{\prime}_{ij},& \mbox{for }i\neq j.
		\end{cases} 
	\end{equation*}
	This means that  we replace each $(i,i)$-th entry $b^{\prime}_{ii}$ of $b$ with $b^{\prime\prime}_{ii}$ to obtain $b_{2}$. It follows that
	\begin{equation*}
		\Phi(b-b_{2})\le \sum_{i\ge 1} \Phi(b^{\prime}_{ii}-b^{\prime\prime}_{ii}) \le \sum_{i\ge 1} \epsilon/2^{i+3} = \epsilon/8. 
	\end{equation*} 
	Furthermore, we express the operator $b_{2}$ in the form $b_{2}=\sum_{i,j}e_{i}b_{2}e_{j}$. At the same time, with respect to $\{e_{i}\}_{i\geq 1}$, we rename $\{\beta^{\prime}_{i,j}\}_{1\le j \le n_i; 1\le i}$ as $\{\beta_{ii}\}_{i\ge 1}$ such that
	\begin{equation*}
		\beta_{ii}e_{i}:=e_{i}b_{2}e_{i}, \quad \mbox{ for all } \ i\ge 1.
	\end{equation*}
	Note that $\beta_{ii}$'s are all real numbers. Apparently, $b_{2}$ is self-adjoint and belongs to $\M$.

	On the other hand, in terms of $\{e_{i}\}_{i\geq 1}$, the operator $a_{1}$ can be expressed as
	\begin{equation*}
		a_{1}:=\sum_{i\geq 1}\alpha^{\prime\prime}_{i}e_{i}.
	\end{equation*}  
	Then as an application of Lemma \ref{compact-perturbation}, there is a diagonal operator $a_{2}:=\sum_{i\geq 1}\alpha_{i}e_{i}$ in $\M$ with a sequence $\{\alpha_{i}\}_{i\geq 1}\subseteq[-2\Vert a\Vert,2\Vert a\Vert]$ such that
	\begin{enumerate}
		\item \quad $\alpha_{i}\neq\alpha_{j}$ for $i\neq j$;
		\item \quad $a_{1}-a_{2}\in{\mathcal{K}^{0}_{\Phi}(\mathcal{M},\tau)^{}}$; 
		\item \quad $\Phi(a_{1}-a_{2})\le\epsilon/8$.
	\end{enumerate}
	Define $x_{2}:=a_{2}+ib_{2}$. We have
	\begin{enumerate}
		\item \quad $x-x_2\in{\mathcal{K}^{0}_{\Phi}(\mathcal{M},\tau)^{}}$; 
		\item \quad $\Phi(x-x_2)\le \Phi(a-a_1)+\Phi(a_1-a_2)+\Phi(b-b_2) \le \epsilon/8+ \epsilon/8+ \epsilon/8< \epsilon/2$. 
	\end{enumerate}

	In the following, we construct an irreducible operator $y$ in $\M$ based on $x_2$ up to a $\Phi(\cdot)$-norm perturbation less than $\epsilon/8$.

	Since $(\M,\tau)$ is a type ${\rm II}_{\infty}$ factor with separable predual and each projection $e_{n}$ is of finite-rank in $(\M,\tau)$, we have that each $e_{n}\M e_{n}$ (restricted to $\mbox{ran}\, e_{n}$) is a type ${\rm II}_{1}$ factor with separable predual. By applying Lemma \ref{Popa}, there exists an irreducible, hyperfinite subfactor $\N_{n}$ in $e_{n}\M e_{n}$. Furthermore, there exists a Cartan masa $\A_{n}$ of $\N_{n}$ such that $\A_{n}$ is also a masa in $e_{n}\M e_{n}$. Recall that a \emph{masa} is a maximal abelian self-adjoint algebra.

	Due to Connes \cite{Connes3}, all the hyperfinite type ${\rm II}_{1}$ factors are isomorphic. By the arguments about the irrational rotation algebra $\A_{\theta}$ preceding Lemma \ref{irreducible-opt}, there exist two unitary operators $u_{n}$ and $v_{n}$ in $\N_{n}$ with $u_{n}v_{n}=e^{2\pi i\theta}v_{n}u_{n}$, generating $\N_{n}$. On the other hand, Lemma \ref{Popa2} guarantees the existence of an automorphism $\pi_n$ on $e_{n}\M e_{n}$ such that $\pi_n(\A_{n})=\{v_{n},v^{\ast}_{n}\}^{\prime\prime}$. For the sake of simplicity, we assume that $\A_{n}=\{v_{n},v^{\ast}_{n}\}^{\prime\prime}$, where by $\{v_{n},v^{\ast}_{n}\}^{\prime\prime}$ we denote the von Neumann algebra generated by $v_n$ and $v^{\ast}_n$ in $e_{n}\M e_{n}$.

	Let $p_{n}$ be a non-trivial projection in $\{u_{n},u^{\ast}_{n}\}^{\prime\prime}$ and $h_{n}$ be a self-adjoint, single generator of $\{v_{n},v^{\ast}_{n}\}^{\prime\prime}$ with $\Vert h_{n}\Vert \leq 1$. Then Lemma \ref{irreducible-opt} entails that $\alpha p_{n}+i\beta h_{n}$ is irreducible in $\N_{n}$ for every pair of non-zero real numbers $\alpha$ and $\beta$.

	Relative to the preceding self-adjoint operator $a_{2}$, there exists a subsequence $\{\lambda_{n}\}_{n\geq 1}$ of $\{\epsilon/2^{n}\}_{n\geq 1}$ such that
	\begin{enumerate}
		\item \quad $\alpha_{n}+\lambda_{n}\neq\alpha_{m}+\lambda_{m}$ for $n\neq m$;
		\item \quad $\sum_{n\geq 1}\lambda_{n}<\epsilon/(8(m_0+1))$.
	\end{enumerate}

	Define $a_{3}:=\sum_{n\geq 1}\alpha_{n}e_{n}+\lambda_{n}p_{n}$. Considering each $\alpha_{n}e_{n}+\lambda_{n}p_{n}$ as an operator in $e_n\M e_n$, it follows that
	\begin{equation*}
		\sigma_{e_n\M e_n}(\alpha_{n}e_{n}+\lambda_{n}p_{n})\cap \sigma_{e_m\M e_m}(\alpha_{m}e_{m}+\lambda_{m}p_{m})=\emptyset, \quad \mbox{ for } n\neq m, 
	\end{equation*}
	and
	\begin{equation*}
		a_2-a_3\in{\mathcal{K}^{0}_{\Phi}(\mathcal{M},\tau)^{}}\quad \mbox{ and }\quad
		\Phi(a_2-a_3)\le\epsilon/8.
	\end{equation*}
	In terms of $\{e_{n}\}_{n\geq 1}$, define $c_3=\sum_{n,m\ge 1}e_{n}c_{3}e_{m}$ to be a self-adjoint operator in $\M$ in the form:
	\begin{equation*}
		e_{n}c_{3}e_{m}=\begin{cases} 0,&\mbox{ if }\ |n-m|\neq 1;\\
		0,& \mbox{ if }\ |n-m|= 1\mbox{ and } e_{n}b_2e_{m}\neq 0;\\
		\lambda_{n}w_{nm},& \mbox{ if }\ n+1=m \mbox{ and } e_{n}b_2e_{m}= 0;\\
		\lambda_{m}w^{\ast}_{mn},& \mbox{ if }\ n-1=m \mbox{ and } e_{n}b_2e_{m}= 0,
		\end{cases} 
	\end{equation*}
	where each $w_{nm}$ is a nonzero partial isometry such that $e_{n}w_{nm}=w_{nm}=w_{nm}e_{m}$. Note that, the existence of nonzero partial isometries $w_{nm}$'s follows from that $\M$ is a factor. Thus $c_{3}\in\M$ is a self-adjoint operator satisfying
	\begin{equation*}
	c_3\in{\mathcal{K}^{0}_{\Phi}(\mathcal{M},\tau)^{}}\quad \mbox{ and }\quad
	\Phi(c_3)\le\epsilon/4. 
	\end{equation*}

	Define $b_{3}:=b_{2}+c_{3}+\sum_{n\geq 1}\lambda_{n}h_{n}$. It follows that 
	\begin{equation*}
		b_3-b_2\in{\mathcal{K}^{0}_{\Phi}(\mathcal{M},\tau)^{}}\quad \mbox{ and }\quad \Phi(b_3-b_2)\le 3\epsilon/8. 
	\end{equation*}

	\vskip 0.3cm
	\noindent\emph{Claim $\ref{II_infty-case}.1.$ \ The operator $y:=a_{3}+ib_{3}$ is irreducible in $\M$.}
	\vskip 0.3cm

	\noindent{\emph{Proof of Claim $\ref{II_infty-case}.1:$}} Actually, if $p$ is a projection in $\M$ commuting with $y$, then $p$ commutes with both $a_{3}$ and $b_{3}$. Write $p=\sum_{n,m\ge 1}p_{nm}$ relative to $\{e_{n}\}_{n\geq 1}$. This means that $p_{nm}:=e_n p e_m$ for each pair of integers $m\ge 1$ and $n\ge 1$. 

	By virtue of the construction of $a_{3}$, we have $p_{nm}=0$ for $n\neq m$. Note that $p_{nn}$ is a subprojection of $e_{n}$ for every $n\ge 1$. It follows that $p_{nn}p_{n}=p_{n}p_{nn}$ and $p_{nn}h_{n}=h_{n}p_{nn}$ for each $n\ge 1$. Since each $h_{n}$ generates $\{v_{n},v^{\ast}_{n}\}^{\prime\prime}$ for all $n\in \NNN$, the projection $p_{nn}$ commutes with each operator in $\{v_{n},v^{\ast}_{n}\}^{\prime\prime}$. Note that $\{v_{n},v^{\ast}_{n}\}^{\prime\prime}$ is also a masa in $e_{n}\M e_{n}$ for every $n\ge 1$. We have $p_{nn}\in \{v_{n},v^{\ast}_{n}\}^{\prime\prime}$ for all $n\ge 1$. Since $p_{n}$ is a non-trivial projection in $\{u_{n},u^{\ast}_{n}\}^{\prime\prime}$ for each $n\ge 1$, the equality $p_{nn}p_{n}=p_{n}p_{nn}$ entails that $p_{nn}$ equals $e_n$ or $0$ in $e_{n}\M e_{n}$ for all $n\geq 1$. 

	Note that $p$ also commutes with each 1-diagonal entry of $b_{3}$. By virtue of the construction of $c_{3}$, it follows that the projection $p$ is trivial in $\M$. This ends the proof of Claim $\ref{II_infty-case}.1$.
	
	\vskip 0.3cm
	
	\noindent\emph{(End of the proof of Theorem $\ref{II_infty-case}$)} Therefore, we have the inequality
	\begin{equation*}
	\begin{aligned}
		\Phi(x-y)&=\Phi(a+ib-a_3-ib_3)\\
		&\leq\Phi(a-a_1)+\Phi(a_1-a_2)+\Phi(a_2-a_3)+\Phi(b-b_2)+\Phi(b_2-b_3)\\
		&\leq \frac{\epsilon}{8}+\frac{\epsilon}{8}+\frac{\epsilon}{8}+\frac{\epsilon}{8}+\frac{3\epsilon}{8}<\epsilon.
	\end{aligned} 
	\end{equation*}
	The construction of $y$ implies that $x-y\in{\mathcal{K}^{0}_{\Phi}(\mathcal{M},\tau)^{}}$. This completes the proof.
\end{proof}

Combining Theorem \ref{B(H)-case} and Theorem \ref{II_infty-case}, we obtain the main theorem of this section.

\begin{theorem}\label{main-thm}
	Suppose that $(\mathcal{M},\tau)$ is a semifinite, properly infinite von Neumann factor with separable predual, where $\tau$ is a faithful, normal, semifinite, tracial weight. Let $\mathcal{K}_{\Phi}(\mathcal{M},\tau)$ be a normed ideal of $(\mathcal{M},\tau)$ equipped with a $\Vert\cdot\Vert$-dominating, unitarily invariant norm $\Phi(\cdot)$ defined as in Definition $\ref{prelim_def1}$. Assume that
	\begin{equation*}
	\displaystyle
	\lim_{\substack{\tau(e)\rightarrow\infty\\ e\in \mathcal{P}\mathcal{F}(\mathcal{M},\tau)}}\frac{\Phi(e)}{\tau(e)} =0. \tag{\ref{not-equiv-to-trace-norm}}
	\end{equation*}
	Then for each $x\in \mathcal{M} $ and every $\epsilon>0$, there exists an irreducible operator $y$ in $\mathcal{M}$ such that
	\begin{equation*}
		x-y\in{\mathcal{K}^{0}_{\Phi}(\mathcal{M},\tau)^{}} \quad \mbox{ and } \quad \Phi(x-y)\le\epsilon.
	\end{equation*}
	In other words, if a $\Vert\cdot\Vert$-dominating, unitarily invariant norm $\Phi(\cdot)$ satisfies $(\ref{not-equiv-to-trace-norm})$, then the set of irreducible operators in $\M$ is $\Phi(\cdot)$-norm dense in $(\M,\tau)$.
\end{theorem}

Note that, in the setting of $\B(\H)$, for every $p\ge 1$, the $\Vert\cdot\Vert_p$-norm is $\Vert\cdot\Vert$-dominating and unitarily invariant. In terms of Remark \ref{rem-B(H)}, if we define $\Phi(\cdot):=\Vert\cdot\Vert_p$, then the condition stated in (\ref{not-equiv-to-trace-norm}) holds if and only if  $p>1$, where $\Vert\cdot\Vert_p$ is defined as in $(\ref{p-norm})$. Thus we have the following corollary.

\begin{corollary}
	Let $\H$ be a complex separable, infinite-dimensional Hilbert space. Let $\mathcal{K}_{p}(\mathcal{H})$ be a normed ideal of $\B(\H)$ equipped with a $\Vert\cdot\Vert_p$-norm satisfying $p>1$ defined as in Definition $\ref{prelim_def2}$.
	Then for each operator $x\in \B(\H)$ and every $\epsilon>0$, there exists an irreducible operator $y$ in $\B(\H)$ such that
	\begin{equation*}
		x-y\in{\mathcal{K}_{p}(\mathcal{H})} \quad \mbox{ and } \quad \Vert x-y \Vert_p\le\epsilon. 
	\end{equation*}
\end{corollary}

Furthermore, in the setting of semifinite von Neumann factors, we can verify that, for every $p>1$, the $\Vert\cdot \Vert$-dominating, unitarily invariant norm $\max \{ \Vert\cdot \Vert_p,\Vert\cdot \Vert \}$ in Definition \ref{prelim_def2} satisfies the condition defined in (\ref{not-equiv-to-trace-norm}). Thus we have the following result.

\begin{corollary}\label{coro-3.7}
	Suppose that $(\mathcal{M},\tau)$ is a semifinite, properly infinite von Neumann factor with separable predual, where $\tau$ is a faithful, normal, semifinite, tracial weight. 
	
	For $p>1$, let $\mathcal{K}_{p}(\mathcal{M},\tau)$ be a normed ideal of $(\mathcal{M},\tau)$ equipped with a $\Vert\cdot \Vert$-dominating, unitarily invariant norm $\max \{ \Vert\cdot \Vert_p,\Vert\cdot \Vert \}$ as in Definition $\ref{prelim_def2}$. 
	
	Then for each operator $x\in \mathcal{M} $ and every $\epsilon>0$, there exists an irreducible operator $y$ in $\mathcal{M}$ such that
	\begin{equation*}
		x-y\in{\mathcal{K}_{p}(\mathcal{M},\tau)} \quad \mbox{ and } \quad \max \{ \Vert x-y \Vert_p,\Vert x-y \Vert \} \le \epsilon. 
	\end{equation*}
	In other words, the set of irreducible operators in $\M$ is $(\max\{ \Vert\cdot \Vert_p,\Vert\cdot \Vert \})$-norm dense in $(\M,\tau)$. 
\end{corollary}

\begin{remark}\label{rem-3.8}
	In Theorem $\ref{main-thm}$, the semifinite factor $(\mathcal{M},\tau)$ is not necessary to be properly infinite. In terms of the same techniques applied in the proof of Theorem $\ref{main-thm}$, The same result for every finite factor with separable predual is also true. The proof for the case of finite factors is similar to that of Theorem $\ref{main-thm}$ and much simpler. 
	
	Note that, if $(\mathcal{M},\tau)$ is a finite factor, then Lemma $2$ of \cite{Fang2} entails that every $ \Vert\cdot \Vert$-dominating, unitarily invariant norm defined on $\M$ is equivalent to the operator norm $ \Vert\cdot \Vert$. In this case, the $\Phi(\cdot)$-norm density of irreducible operators in $\M$ can also be proved by \cite[Theorem 2.1]{Fang}. In this point of view, Theorem $\ref{main-thm}$ is more interesting in semifinite, properly infinite factors  $(\mathcal{M},\tau)$, since the normed ideal $\mathcal{K}_{\Phi}(\mathcal{M},\tau)$ is nontrivial in general.
\end{remark}

\section{Density of irreducible operators up to normed ideal perturbations in the set of  normal operators in semifinite von Neumann factors}

It is worth mentioning that, while studying the density of irreducible operators up to the normed ideal perturbation in semifinite von Neumann factors, an important technique is Theorem \ref{Li-Shen-Shi-3}, which is an extended Weyl-von Neumann theorem for self-adjoint operators in the setting of semifinite von Neumann algebras. The condition stated in (\ref{not-equiv-to-trace-norm}) is crucial to Theorem \ref{Li-Shen-Shi-3}. Thus it is interesting to ask Problem \ref{Problem-1.2}, i.e., whether irreducible operators are $\Phi(\cdot)$-norm dense in each semifinite von Neumann factor with separable predual.  

Let $\H$ be a separable Hilbert space. In the setting of $\B(\H)$, by virtue of the Kato-Rosenblum's theorem, a self-adjoint operator $a\in\B(\H)$ with a non-vanishing spectrally absolutely continuous part, can't be expressed as a diagonal operator plus an arbitrarily small trace norm perturbation. This implies that, to study the density of irreducible operators in each semifinite von Neumann factor $(\M,\tau)$ with respect to the $\{\Vert\cdot\Vert,\Vert\cdot\Vert_1\}$-norm, it is necessary to develop new techniques.

In this section, we introduce a new approach to prove that each normal operator in a semifinite von Neumann factor $(\M,\tau)$ can be expressed as an irreducible operator in $\M$ plus an arbitrarily small $\Phi(\cdot)$-norm perturbation, where the $\Phi(\cdot)$-norm induces a normed ideal $\K_{\Phi}(\M,\tau)$ as defined in Definition $\ref{prelim_def1}$. With this new approach, the restriction in $(\ref{not-equiv-to-trace-norm})$ can be removed from the $\Phi(\cdot)$-norm in the above perturbation.

\begin{theorem}\label{normal-to-irreducible}
	Suppose that $(\mathcal{M},\tau)$ is a semifinite von Neumann factor with separable predual, where $\tau$ is a faithful, normal, semifinite, tracial weight. Let $\mathcal{K}_{\Phi}(\mathcal{M},\tau)$ be a normed ideal of $(\mathcal{M},\tau)$ equipped with a $\Vert\cdot\Vert$-dominating, unitarily invariant norm $\Phi(\cdot)$ $(\mbox{see Definition } \ref{prelim_def1})$. 
	
	For each normal operator $x$ in $ \M$ and every $\epsilon>0$, there is an irreducible operator $y$ in $\mathcal{M}$ such that
	\begin{equation*}
		x-y\in{\mathcal{K}^{0}_{\Phi}(\mathcal{M},\tau)^{}} \quad \mbox{ and } \quad \Phi(x-y)\le\epsilon.
	\end{equation*}
\end{theorem}

\begin{remark}
	As mentioned in Remark $\ref{rem-3.8}$, for a finite factor, every $\Vert\cdot \Vert$-dominating, unitarily invariant norm $\Phi(\cdot)$ is equivalent to the operator norm $\Vert\cdot\Vert$. Thus the proof, when $(\mathcal{M},\tau)$ is a finite von Neumann factor with separable predual, is an application of \cite[Theorem 2.1]{Fang}.

Therefore, it is sufficient to prove Theorem $\ref{normal-to-irreducible}$ when $(\mathcal{M},\tau)$ is an infinite von Neumann factor with separable predual. 
\end{remark}

Recall that, for each operator $x$ in a ${\ast}$-algebra $\A$,  $\mbox{Re}\, x$ (the real part of $x$) and $\mbox{Im}\, x$ (the imaginary part of $x$) are defined as 
\begin{equation}
	\mbox{Re}\, x:=\frac{x+x^\ast}{2} \quad \mbox{ and } \quad \mbox{Im}\, x:=\frac{x-x^\ast}{2i}.
\end{equation}
It follows that $x=\mbox{Re} \, x + i \, \mbox{Im} \, x$, i.e., each operator $x$ can be expressed as a sum of self-adjoint operators in $\A$.

In the following, we prepare three lemmas for Lemma \ref{normal-case-1}. In fact,  operators mentioned in the following three lemmas aren't  necessary to be normal.

\begin{lemma}\label{lemma-case-1}
	Suppose that $(\mathcal{M},\tau)$ is a semifinite, properly infinite von Neumann factor with separable predual, where $\tau$ is a faithful, normal, semifinite, tracial weight. Let $\mathcal{K}_{\Phi}(\mathcal{M},\tau)$ be a normed ideal of $(\mathcal{M},\tau)$ equipped with a $\Vert\cdot\Vert$-dominating, unitarily invariant norm $\Phi(\cdot)$ $(\mbox{see Definition } \ref{prelim_def1})$.
	 
	Suppose that $x$ is an operator in $ \M$. If there exist three infinite, spectral projections $\{p_j\}^{3}_{j=1}$ for ${\rm Re} \, x$ with 
	\begin{equation}\label{eq-lem-case-1}
		p_1+p_2+p_3=I,\quad \mbox{ and } \quad p_j \cdot ({\rm Im} \, x)=({\rm Im} \, x) \cdot  p_j, \quad  \mbox{ for all }\ j=1,2,3,
	\end{equation} 
	then for every $\epsilon>0$, there exists an irreducible operator $y$ in $\M$ such that
	\begin{equation}\label{condition-lem-1}
		x-y\in{\mathcal{K}^{0}_{\Phi}(\mathcal{M},\tau)^{}} \quad \mbox{ and } \quad \Phi(x-y)\le\epsilon.
	\end{equation}    
\end{lemma}

\begin{proof}
	First, we apply systems of matrix units for $\M$ to construct an operator $y$ in $\M$. Then we prove that $y$ is irreducible in $\M$ and it satisfies (\ref{condition-lem-1}). 
	
	Since $\M$ is a semifinite, properly infinite factor, there exists a system of matrix units $\{p_{ij}\}^{3}_{i,j =1}$ for $\M$ such that
	\begin{enumerate}
		\item \quad $p_{jj}=p_{j}$, for all $j=1,2,3$,
		\item \quad $p^*_{ji}=p_{ij}$ for all  $i, j=1,2,3$,
		\item \quad $p_{ij}p_{kl}=\delta_{jk}p_{il}$ for each  $i, j, k, l=1,2,3$,
	\end{enumerate}
	where $\delta_{jk}$ means the Kronecker symbol.
	
	Furthermore, there exists a system of matrix units $\{e_{ij}\}_{i,j\in\NNN}$ for $p_{11} \M p_{11}$ satisfying 
	\begin{enumerate}
		\item \quad  {\scriptsize {SOT}}-$\sum^{}_{i\in\NNN}e_{ii}=p_{11}$;
		\item \quad  $e^*_{ji}=e_{ij}$ for all  $i, j\in\NNN$;
		\item \quad  $e_{ij}e_{kl}=\delta_{jk}e_{il}$ for all  $i, j, k, l\in\NNN$;
		\item \quad  $\{e_{ij}\}_{i,j\in\NNN}\subset \F(\M,\tau)$ and by virtue of Lemma \ref{prelim_lemma1}, there is a uniform upper bound $\eta>0$ such that $\Phi(e_{ij})\leq \eta$ for all $i,j\in\NNN$.
	\end{enumerate}
	
	Define $\P$ to be the von Neumann algebra generated by $\{e_{ij}\}_{i,j\in\NNN}\cup \{p_{ij}\}^{3}_{i,j=1}$. It follows that $\P$ is a type ${\rm I}_{\infty}$ subfactor of $\M$, which is $\ast$-isomorphic to $\B(l^2(\NNN))$. Define $\N=\P^{\prime}\cap \M$. Then $\N$ is a finite subfactor of $\M$, which is $\ast$-isomorphic to $e_{11} \M e_{11}$ and $(\N\cup \P)^{\prime\prime}=\M$. It follows that $\M$ is $\ast$-isomorphic to the von Neumann algebra tensor product $\mathcal P  \otimes \mathcal N$, which is denoted by $\mathcal M\cong \mathcal P \otimes \mathcal N$.

	\vskip 0.3cm
	\noindent\emph{Claim $\ref{lemma-case-1}.1.$ \ There exist two invertible positive operators $a$ and $b$ in $\N$ such that $a+ib$ is irreducible in $\N$. }
	\vskip 0.3cm

	\noindent{\emph{Proof of Claim $\ref{lemma-case-1}.1:$}} Since $\N$ is a finite subfactor of $\M$, we have that $\N$ is either a type ${\rm I}_n$ factor for some $n\in\NNN$ or $\N$ is a type ${\rm II}_1$ subfactor of $\M$.  If $\N$ is a finite type I subfactor of $\M$, then we can use linear algebra techniques to choose a single generator $a+ib$ of $\N$ such that both $a$ and $b$ are invertible and positive. Otherwise, if $\N$ is a type ${\rm II}_1$ subfactor of $\M$, by  Lemma \ref{Popa} there exists an irreducible, hyperfinite type ${\rm II}_1$ subfactor $\R\subseteq\N$. Note that each hyperfinite type ${\rm II}_1$ factor is singly generated (see \cite[Theorem 1]{Saito_2}). It follows that, we can choose $a+i b$ to be a single generator of $\R$, where $a$ and $b$ are invertible, positive operators in $\R$. In each case, $a+ib$ is irreducible in $\N$. This ends the proof of Claim $\ref{lemma-case-1}.1$.

	\vskip 0.3cm

	\noindent {(\emph{End of the proof of Lemma $\ref{lemma-case-1}$})} Define a self-adjoint operator $h$ in $\M$ of the form
	 \begin{equation*}
	 \begin{aligned}
	 	h:=& \mbox{Im}\, x+\frac{\epsilon}{ 3^2 \cdot \eta} \sum\nolimits^{\infty}_{j=1} \frac{1}{ 3^{j}} (e_{jj}p_{12}+p_{21}e_{jj})+ \frac{\epsilon\cdot a}{ 3^2 \cdot  \eta\cdot \Vert a \Vert} \sum\nolimits^{\infty}_{j=1} \frac{1}{ 3^{j}} (e_{jj}p_{13}+p_{31}e_{jj})\\
	 	& +\frac{\epsilon\cdot b }{ 3^2 \cdot  \eta\cdot \Vert b \Vert} \sum\nolimits^{\infty}_{j=1} \frac{1}{ 3^{j}}  (p_{21}e_{1j}p_{13}+p_{31}e_{j1}p_{12}).
	 \end{aligned}
	 \end{equation*}
	 As an application of Lemma \ref{prelim_lemma1}, it follows that $\Phi(\mbox{Im}\, x -h) \le \epsilon/3$.
	 
	 Define $y=\mbox{Re}\, x+i h$.  Thus $\Phi(x-y)< \epsilon$.  To prove that $y$ is irreducible in $\M$, it is sufficient to show that
	 \begin{equation*}
	 	 \{e_{ij}\}_{i,j\in\NNN}\cup \{p_{ij}\}^{3}_{i,j=1} \cup \{a,b\}\subset  W^{\ast}(y),
	 \end{equation*}
	 where $W^{\ast}(y)$ is the von Neumann sub-algebra of $\M$ generated by $y$ and the identity  $I$ of $\M$.
	 
	 In terms of (\ref{eq-lem-case-1}), we have
	 \begin{equation*}
	 	\{p_{jj}\}^3_{j=1}\subset W^{\ast}(\mbox{Re}\,  x)  \subset W^{\ast}(y).
	 \end{equation*}
	 Since $p_{jj} \cdot ({\rm Im} \, x)=({\rm Im} \, x) \cdot  p_{jj}$  for all  $j=1,2,3$, it follows that
	 \begin{equation*}
	 	p_{11}hp_{22}=\frac{\epsilon}{3^2 \cdot  \eta} \sum\nolimits^{+\infty}_{j=1} \frac{1}{ 3^{j}}  (e_{jj}p_{12}) \in  W^{\ast}(y).
	 \end{equation*}
	 This entails that 
	 \begin{equation*}
	 	p_{11}hp_{22}hp_{11}=\frac{\epsilon^2}{3^4 \cdot  \eta^2}  \sum\nolimits^{+\infty}_{j=1} \frac{1}{ 3^{(2j)}}  e_{jj} \in  p_{11} W^{\ast}(y) p_{11}.
	 \end{equation*}
	 A function calculus of the positive operator $p_{11}hp_{22}hp_{11}$ yields that
	 \begin{equation*}
	 	\{e_{jj}\}_{j\in \NNN}  \subseteq p_{11} W^{\ast}(y) p_{11}\subseteq  W^{\ast}(y).
	 \end{equation*}
	 Thus for all $j\in \NNN$, we have $e_{jj}hp_{22}=e_{jj}p_{12}\in  W^{\ast}(y)$. It follows that 
	 \begin{equation*}
	 	p_{12}=\mbox{\scriptsize{SOT}-}\sum\nolimits^{\infty}_{j=1}e_{jj}p_{12} \in  W^{\ast}(y).
	 \end{equation*}
	 Note that  for all $j\in \NNN$,  
	 \begin{equation*}
	 	e_{jj}hp_{33}=\frac{\epsilon \cdot a e_{jj}p_{13}}{ 3^{j+2} \cdot \eta \cdot \Vert a \Vert} \in W^{\ast}(y).
	 \end{equation*} 
	 Since $a$ is positive and invertible in $\N$, the uniqueness of the polar decomposition of $ae_{jj}p_{13}$ entails that $e_{jj}p_{13}$ is in $W^{\ast}(y)$ for all $j\in\NNN$. It follows that 
	 $$p_{13}=\mbox{{\scriptsize {SOT}}-}\sum\nolimits_{j\in\NNN}e_{jj}p_{13}\in W^{\ast}(y).$$
	 Thus, the fact that $W^{\ast}(y)$ is a $\ast$-algebra entails that $\{p_{21}, p_{31}\}\subset W^*(y)$. Therefore, we have 
	 \begin{equation*}
	 	\{p_{ij}\}^{3}_{i,j=1} \subset W^{\ast}(y).
	 \end{equation*}
	 Since  $\P$ is the von Neumann algebra generated by $\{e_{ij}\}_{i,j\in\NNN}\cup \{p_{ij}\}^{3}_{i,j=1}$, we have $\P \subseteq W^{\ast}(y)$. This implies that,  for all $j\in \NNN$ and $1\le k \le 3$, 
	 \begin{equation*}
	 	ap_{k1}e_{jj}p_{1k}=p_{k1}ae_{jj}p_{13}p_{3k} \in W^{\ast}(y).
	 \end{equation*}
	 It follows that 
	 \begin{equation*}
	 	a=\mbox{\scriptsize{SOT}-}\sum\nolimits^{3}_{k=1}\sum\nolimits^{\infty}_{j=1}ap_{k1}e_{jj}p_{1k} \in  W^{\ast}(y).
	 \end{equation*}
	 A similar calculation shows that 
	 \begin{equation*}
	 	b=\mbox{\scriptsize{SOT}-}\sum\nolimits^{3}_{k=1}\sum\nolimits^{\infty}_{j=1}bp_{k1}e_{jj}p_{1k}=\mbox{\scriptsize{SOT}-}\sum\nolimits^{3}_{k=1}\sum\nolimits^{\infty}_{j=1}bp_{k3}(p_{31}e_{j1}p_{12})p_{21}e_{1j}p_{1k} \in  W^{\ast}(y).
	 \end{equation*}
	 Therefore, we have
	  \begin{equation*}
	 	\P \cup \{a,b\} \subset  W^{\ast}(y).
	 \end{equation*}
	 
	 We are ready to prove that $y$ is irreducible in $\M$. Suppose that $q$ is a projection in $\M$ commuting with $y$. It follows that $q$ commutes with every element in $\P$. Thus we have $q \in \N$. Note that $qa=aq$ and $qb=bq$. In terms of Claim $\ref{lemma-case-1}.1$, the operator $a+ib$ is irreducible in $\N$. This entails that $q=0$ or $q=I$. Therefore, $y=\mbox{Re}\, x+ih$ is irreducible in $\M$. This completes the proof.
\end{proof}

\begin{lemma}\label{lemma-case-2}
	Suppose that $(\mathcal{M},\tau)$ is a semifinite, properly infinite von Neumann factor with separable predual, where $\tau$ is a faithful, normal, semifinite, tracial weight. Let $\mathcal{K}_{\Phi}(\mathcal{M},\tau)$ be a normed ideal of $(\mathcal{M},\tau)$ equipped with a $\Vert\cdot\Vert$-dominating, unitarily invariant norm $\Phi(\cdot)$ $(\mbox{see Definition } \ref{prelim_def1})$.
	 
	Suppose that $x$ is an operator in $ \M$. If ${\rm Re} \, x $ is in the form 
	\begin{equation*}
		{\rm Re} \, x=\lambda_1 e_1+\lambda_2 e_2+(I-e_1-e_2) {\rm Re} \, x 
	\end{equation*}
	such that
	\begin{enumerate}
		\item \quad  $\lambda_1\neq \lambda_2$,
		\item \quad  $e_1$ and $e_2$  are infinite, mutually orthogonal, spectral projections for $\mbox{Re} \, x $,
		\item \quad  $e_j \cdot ({\rm Im} \, x )=({\rm Im} \, x ) \cdot e_j$ for $j=1,2$,
		\item \quad  $0<\tau(I-e_1-e_2)<\infty$,
	\end{enumerate}	
	then for every $\epsilon>0$, there exists an irreducible operator $y$ in $\M$ such that
	\begin{equation}
		x-y\in{\mathcal{K}^{0}_{\Phi}(\mathcal{M},\tau)^{}} \quad \mbox{ and } \quad \Phi(x-y)\le\epsilon. \tag{\ref{condition-lem-1}}
	\end{equation}    
\end{lemma}

\begin{proof}
	Since $\M$ is a semifinite, properly infinite factor, the induced von Neumann algebra $(e_1+e_2)\M(e_1+e_2)$ is also a semifinite, properly infinite factor. Thus there exists a system of matrix units $\{p_{ij}\}^{2}_{i,j =1}$ for $(e_1+e_2)\M(e_1+e_2)$ such that
	\begin{enumerate}
		\item \quad $p_{jj}=e_j$ is infinite in $(e_1+e_2)\M(e_1+e_2)$ for $j=1,2$,
		\item \quad $p^*_{ji}=p_{ij}$ for all  $i, j=1, 2$,
		\item \quad $p_{ij}p_{kl}=\delta_{jk}p_{il}$ for each  $i, j, k, l=1,2$,
	\end{enumerate}
	where $\delta_{jk}$ means the Kronecker symbol.
	
	Furthermore, there exists a system of matrix units $\{e_{i j}\}^{\infty}_{i,j=0}$ for $(I-e_2)\M(I-e_2)$ satisfying 
	\begin{enumerate}
		\item\quad  $e_{00}=I-e_1-e_2$ and  {\scriptsize {SOT}}-$\sum^{\infty}_{j=1}e_{jj}=p_{11}=e_1$;
		\item\quad  $e^*_{ji}=e_{ij}$ for all  $i, j\ge 0$;
		\item\quad  $e_{ij}e_{kl}=\delta_{jk}e_{il}$ for all  $i, j, k, l\ge 0$;
		\item\quad  $\{e_{ij}\}_{i,j\ge 0}\subset \F(\M,\tau)$ and by virtue of Lemma \ref{prelim_lemma1}, there is a uniform upper bound $\eta>0$ such that $\Phi(e_{ij})\leq \eta$ for all $i,j\ge 0$.
	\end{enumerate}
	
	Define $\P$ to be the von Neumann algebra generated by $\{e_{ij}\}_{i,j\ge 0}\cup \{p_{ij}\}^{2}_{i,j=1}$. It follows that $\P$ is a type ${\rm I}_{\infty}$ subfactor of $\M$, which is $\ast$-isomorphic to $\B(l^2(\NNN))$. Define $\N=\P^{\prime}\cap \M$. Then $\N$ is a finite subfactor of $\M$, which is $\ast$-isomorphic to $e_{11} \M e_{11}$ and $(\N\cup \P)^{\prime\prime}=\M$. It follows that $\M$ is $\ast$-isomorphic to the von Neumann algebra tensor product $\mathcal P  \otimes \mathcal N$, which is denoted by $\mathcal M\cong \mathcal P \otimes \mathcal N$.
	
	By virtue of Claim \ref{lemma-case-1}.1, there exist two invertible positive operators $a$ and $b$ in $\N$ such that $a+ib$ is irreducible in $\N$. Define a self-adjoint operator $h$ in $\M$ of the form
	\begin{equation*}
	\begin{aligned}
		h:=& \mbox{Im}\, x+\frac{\epsilon \cdot a}{ 3^2 \cdot \eta \cdot \Vert a \Vert} \sum\nolimits^{\infty}_{j=1} \frac{1}{ 3^{j}} (e_{0j}+e_{j0})+ \frac{\epsilon  \cdot b}{ 3^2 \cdot  \eta\cdot \Vert b \Vert} \sum\nolimits^{\infty}_{j=1} \frac{1}{ 3^{j}} (e_{0j}p_{12}+p_{21}e_{j0})\\
	 	& +\frac{\epsilon }{ 3^2 \cdot  \eta} \sum\nolimits^{\infty}_{j=1} \frac{1}{ 3^{j}}  (e_{jj}p_{12}+p_{21}e_{jj}).
	\end{aligned}
	\end{equation*}
	As an application of Lemma \ref{prelim_lemma1}, it follows that $\Phi(\mbox{Im}\, x -h) \le \epsilon/3$.
	 
	Define $y=\mbox{Re}\, x+i h$.  Thus $\Phi(x-y)< \epsilon$.  To prove that $y$ is irreducible in $\M$, it is sufficient to show that
	\begin{equation*}
	 	 \{e_{ij}\}_{i,j \ge 0}\cup \{p_{ij}\}^{2}_{i,j=1} \cup \{a,b\}\subset  W^{\ast}(y).
	\end{equation*}
	 
	Note that $e_1$ and $e_2$ are spectral projections of $\mbox{Re}\,  x$. This entails that
	\begin{equation*}
		\{e_{00}, p_{11}, p_{22}\} \subset W^{\ast}(\mbox{Re}\,  x)  \subset W^{\ast}(y).
	\end{equation*}
	Since $p_{jj} \cdot ({\rm Im} \, x )=({\rm Im} \, x ) \cdot p_{jj} $ for $j=1,2$, it follows that
	\begin{equation*}
		p_{11}hp_{22}=\frac{\epsilon}{3^2 \cdot  \eta} \sum\nolimits^{+\infty}_{j=1} \frac{1}{ 3^{j}}  (e_{jj}p_{12}) \in  W^{\ast}(y).
	\end{equation*}
	A similar argument as in Lemma \ref{lemma-case-1} entails that 
	\begin{equation*}
	 	\{e_{jj}\}_{j\ge 0} \cup \{p_{ij}\}^{2}_{i,j=1} \subset  W^{\ast}(y).
	\end{equation*}
	Note that  for all $j\ge 1$, 
	\begin{equation*}
		e_{00}he_{jj}=\frac{\epsilon \cdot a \cdot e_{0j}}{ 3^{j+2} \cdot \eta \cdot \Vert a \Vert} \in W^{\ast}(y).
	\end{equation*} 
	Since $a$ is positive and invertible in $\N$, the uniqueness of the polar decomposition of $ae_{0j}$ entails that $e_{0j}$ is in $W^{\ast}(y)$ for all $j\ge 1$. It follows that 
	$$\{e_{ij}\}_{i,j\ge 0} \cup \{p_{ij}\}^{2}_{i,j=1} \subset  W^{\ast}(y).$$
	Thus, we have $\P \subseteq W^{\ast}(y)$. This implies that,  for all $j\ge 1$, 
	\begin{equation*}
		ae_{00} \in W^{\ast}(y), \quad ae_{jj} \in W^{\ast}(y), \quad \mbox{ and } \quad  ap_{21}e_{jj}p_{12} \in W^{\ast}(y).
	\end{equation*}
	It follows that 
	\begin{equation*}
		a=\mbox{\scriptsize{SOT}-}\sum\nolimits^{\infty}_{j=0}ae_{jj} +\mbox{\scriptsize{SOT}-}\sum\nolimits^{\infty}_{j=1}ap_{21}e_{jj}p_{12} \in  W^{\ast}(y).
	\end{equation*}
	A similar calculation shows that $b\in  W^{\ast}(y)$.
	Therefore, we have $\P \cup \{a,b\} \subset  W^{\ast}(y)$.
	
	We are ready to prove that $y$ is irreducible in $\M$. A similar argument as in Lemma \ref{lemma-case-1} entails that  $y=\mbox{Re}\, x+ih$ is irreducible in $\M$. This completes the proof.
\end{proof}

\begin{lemma}\label{lemma-case-3}
	Suppose that $(\mathcal{M},\tau)$ is a semifinite, properly infinite von Neumann factor with separable predual, where $\tau$ is a faithful, normal, semifinite, tracial weight. Let $\mathcal{K}_{\Phi}(\mathcal{M},\tau)$ be a normed ideal of $(\mathcal{M},\tau)$ equipped with a $\Vert\cdot\Vert$-dominating, unitarily invariant norm $\Phi(\cdot)$ $(\mbox{see Definition } \ref{prelim_def1})$.
	 
	Suppose that $x$ is an operator in $ \M$. If ${\rm Re} \, x $ is in the form 
	\begin{equation*}
		{\rm Re} \, x=\lambda_1 e_1+\lambda_2 e_2
	\end{equation*}
	such that
	\begin{enumerate}
		\item \quad  $\lambda_1$ and $ \lambda_2$ are real numbers with $\lambda_1\neq \lambda_2$,
		\item \quad  $e_1$ and $e_2$  are infinite,  spectral projections for ${\rm Re} \, x $ with $e_1+e_2=I$,
		\item \quad  $e_j \cdot ({\rm Im} \, x )=({\rm Im} \, x ) \cdot e_j$ for $j=1,2$,
	\end{enumerate}
	then for every $\epsilon>0$, there exists an irreducible operator $y$ in $\M$ such that
	\begin{equation}
		x-y\in{\mathcal{K}^{0}_{\Phi}(\mathcal{M},\tau)^{}} \quad \mbox{ and } \quad \Phi(x-y)\le\epsilon. \tag{\ref{condition-lem-1}}
	\end{equation}    
\end{lemma}

\begin{proof}
	Since $\M$ is a semifinite, properly infinite factor, there exists a system of matrix units $\{p_{ij}\}^{3}_{i,j =1}$ for $\M$ such that
	\begin{enumerate}
		\item \quad $p_{11}=e_1$ and $p_{22}+p_{33}=e_2$,
		\item \quad $p_{jj}$ is infinite in $\M$ for $j=1,2,3$,
		\item \quad $p_{jj} \cdot ({\rm Im} \, x )=({\rm Im} \, x ) \cdot p_{jj}$ for $j=1,2,3$,
		\item \quad $p^*_{ji}=p_{ij}$ for all  $i, j=1,2,3$,
		\item \quad $p_{ij}p_{kl}=\delta_{jk}p_{il}$ for each  $i, j, k, l=1,2,3$,
	\end{enumerate}
	where $\delta_{jk}$ means the Kronecker symbol. 
		
	Furthermore, there exists a system of matrix units $\{e_{ij}\}_{i,j\in\NNN}$ for $p_{33} \M p_{33}$ satisfying 
	\begin{enumerate}
		\item \quad   {\scriptsize {SOT}}-$\sum^{}_{i\in\NNN}e_{ii}=p_{33}$;
		\item \quad $e^*_{ji}=e_{ij}$ for all  $i, j\in\NNN$;
		\item \quad  $e_{ij}e_{kl}=\delta_{jk}e_{il}$ for all  $i, j, k, l\in\NNN$;
		\item \quad  $\{e_{ij}\}_{i,j\in\NNN}\subset \F(\M,\tau)$ and by virtue of Lemma \ref{prelim_lemma1}, there is a uniform upper bound $\eta>0$ such that $\Phi(e_{ij})\leq \eta$ for all $i,j\in\NNN$.
	\end{enumerate}
	
	Define a self-adjoint operator $d$ in $\M$ of the following form
	\begin{equation*}
		d:=\mbox{Re} \, x+\sum\nolimits^{\infty}_{j=1}\frac{\epsilon\cdot e_{jj}}{3^{j}\cdot \eta}.
	\end{equation*}
	It follows that $\Phi(\mbox{Re} \, x-d)\le \epsilon /2$. Note that $d+i\, \mbox{Im} \, x$ satisfies the assumption of Lemma \ref{lemma-case-1}. In terms of the proof of Lemma \ref{lemma-case-1}, for $d+i\, \mbox{Im} \, x$, there exists an irreducible operator $y$ in $\M$ such that 
	\begin{equation*}
		\Phi(d+i\, \mbox{Im} \, x - y)\le \frac{\epsilon }{3}.
	\end{equation*}
	Therefore, we have
	\begin{equation*}
		\Phi( x - y)\le \Phi(\mbox{Re} \, x-d) + \Phi ( d+i\, \mbox{Im} \, x -y)\le \frac{\epsilon }{2} +\frac{\epsilon }{3} < \epsilon .
	\end{equation*}
	This completes the proof.
\end{proof}

We are ready to prove Lemma \ref{normal-case-1}.

\begin{lemma}\label{normal-case-1}
	Suppose that $(\mathcal{M},\tau)$ is a semifinite, properly infinite von Neumann factor with separable predual, where $\tau$ is a faithful, normal, semifinite, tracial weight. Let $\mathcal{K}_{\Phi}(\mathcal{M},\tau)$ be a normed ideal of $(\mathcal{M},\tau)$ equipped with a $\Vert\cdot\Vert$-dominating, unitarily invariant norm $\Phi(\cdot)$ $(\mbox{see Definition } \ref{prelim_def1})$.
	 
	Suppose that $x$ is a normal operator in $ \M$. If there exist three infinite, spectral projections $\{p_j\}^{3}_{j=1}$ for $x$ with 
	\begin{equation*}
		p_1+p_2+p_3=I,
	\end{equation*} 
	then for every $\epsilon>0$, there exists an irreducible operator $y$ in $\M$ such that
	\begin{equation}
		x-y\in{\mathcal{K}^{0}_{\Phi}(\mathcal{M},\tau)^{}} \quad \mbox{ and } \quad \Phi(x-y)\le\epsilon. \tag{\ref{condition-lem-1}}
	\end{equation}    
\end{lemma}

\begin{proof}
	Note that if there are three infinite, spectral projections $\{p_j\}^{3}_{j=1}$ for a normal operator $x$ in $\M$ with $p_1+p_2+p_3=I$, then at least one of the following three cases  happens: 
	\begin{enumerate}
		\item [{\bf Case 1}:] \quad for either  $\mbox{Re} \, x $ or $ \mbox{Im} \, x$, there are three, infinite, spectral projections with sum $I$;
		\item [{\bf Case 2}:] \quad either	$\mbox{Re} \, x $ or $ \mbox{Im} \, x$ is in  the  form
		\begin{equation*}
			\mbox{Re} \, x=\lambda_1 e_1+\lambda_2 e_2+(I-e_1-e_2)\mbox{Re} \, x \quad \mbox{ or } \quad \mbox{Im} \, x=\eta_1 f_1+\eta_2 f_2+(I-f_1-f_2)\mbox{Im} \, x
		\end{equation*}
		where $\lambda_1\neq \lambda_2$ (resp. $\eta_1\neq \eta_2$), $e_1$ and $e_2$ (resp. $f_1$ and $f_2$) are infinite, mutually orthogonal, spectral projections for $\mbox{Re} \, x$ (resp. $\mbox{Im} \, x $) such that $I-e_1-e_2$ (resp. $I-f_1-f_2$) is finite and nonzero;
		\item [{\bf Case 3}:] \quad	either $\mbox{Re} \, x $ or $ \mbox{Im} \, x$ is in the form
		\begin{equation*}
			\mbox{Re} \, x=\lambda_1 e_1+\lambda_2 e_2 \quad \mbox{ or } \quad \mbox{Im} \, x=\eta_1 f_1+\eta_2 f_2
		\end{equation*}
		where $\lambda_1 $ and $\lambda_2$ (resp. $\eta_1$ and $\eta_2$)  are real numbers with $\lambda_1\neq \lambda_2$ (resp. $\eta_1\neq \eta_2$), and the projections $e_1$ and $e_2$ (resp. $f_1$ and $f_2$) are infinite, spectral projections for $\mbox{Re} \, x$ (resp. $\mbox{Im} \, x $) with $e_1+e_2=I$ (resp. $f_1+f_2=I$).
	\end{enumerate}
	
	By virtue of Lemma \ref{lemma-case-1}, we prove {\bf Case 1}. By virtue of Lemma \ref{lemma-case-2}, we prove {\bf Case 2}. By virtue of Lemma \ref{lemma-case-3}, we prove {\bf Case 3}. This completes the proof.
	\end{proof}

\begin{lemma}\label{normal-case-2}
	Suppose that $(\mathcal{M},\tau)$ is a semifinite, properly infinite von Neumann factor with separable predual, where $\tau$ is a faithful, normal, semifinite, tracial weight. Let $\mathcal{K}_{\Phi}(\mathcal{M},\tau)$ be a normed ideal of $(\mathcal{M},\tau)$ equipped with a $\Vert\cdot\Vert$-dominating, unitarily invariant norm $\Phi(\cdot)$ $(\mbox{see Definition } \ref{prelim_def1})$. 
	
	Suppose that $x$ is a normal operator in $ \M$. If there are at most two infinite, mutually orthogonal, spectral projections for $ x$, then for every $\epsilon>0$, there is an irreducible operator $y$ in $\mathcal{M}$ such that
	\begin{equation*}
		x-y\in{\mathcal{K}^{0}_{\Phi}(\mathcal{M},\tau)} \quad \mbox{ and } \quad \Phi(x-y)\le\epsilon. \tag{\ref{condition-lem-1}}
	\end{equation*}
\end{lemma}

\begin{proof}
	Note that if there are at most two infinite, mutually orthogonal, spectral projections for the normal operator $x$ in $\M$, then one of the following four cases must happen: 
	\begin{enumerate}
		\item [{\bf Case 1}:] \quad	either $\mbox{Re} \, x $ or $ \mbox{Im} \, x$ is in the form
		\begin{equation*}
			\mbox{Re} \, x=\lambda_1 e_1+\lambda_2 e_2+(I-e_1-e_2)\mbox{Re} \, x \quad \mbox{ or } \quad \mbox{Im} \, x=\eta_1 f_1+\eta_2 f_2+(I-f_1-f_2)\mbox{Im} \, x
		\end{equation*}
		where $\lambda_1 $ and $\lambda_2$ (resp. $\eta_1$ and $\eta_2$)  are real numbers with $\lambda_1\neq \lambda_2$ (resp. $\eta_1\neq \eta_2$), and the projections $e_1$ and $e_2$ (resp. $f_1$ and $f_2$) are infinite, mutually orthogonal, spectral projections for $\mbox{Re} \, x$ (resp. $\mbox{Im} \, x $) such that $I-e_1-e_2$ (resp. $I-f_1-f_2$) is finite and nonzero;
		\item [{\bf Case 2}:] \quad	either $\mbox{Re} \, x $ or $ \mbox{Im} \, x$ is in the form
		\begin{equation*}
			\mbox{Re} \, x=\lambda_1 e_1+\lambda_2 e_2 \quad \mbox{ or } \quad \mbox{Im} \, x=\eta_1 f_1+\eta_2 f_2
		\end{equation*}
		where $\lambda_1\neq \lambda_2$ (resp. $\eta_1\neq \eta_2$), $e_1$ and $e_2$ (resp. $f_1$ and $f_2$) are infinite, spectral projections for $\mbox{Re} \, x$ (resp. $\mbox{Im} \, x $) with $e_1+e_2=I$ (resp. $f_1+f_2=I$);
		\item [{\bf Case 3}:] \quad $\mbox{Re} \, x $ and $ \mbox{Im} \, x$ are in the forms
		\begin{equation*}
			\mbox{Re} \, x=\lambda e+(I-e)\mbox{Re} \, x \quad \mbox{ and } \quad \mbox{Im} \, x=\eta f+(I-f)\mbox{Im} \, x
		\end{equation*}
		where $e$  (resp. $f$) is an infinite, spectral projection for $\mbox{Re} \, x$ (resp. $\mbox{Im} \, x $) with $I-e$ (resp. $I-f$) finite and nonzero;
		\item [{\bf Case 4}:] \quad	either $\mbox{Re} \, x $ or $ \mbox{Im} \, x$ is a scalar multiple of the identity.
	\end{enumerate}
	We prove this lemma in the above four cases.
	
	By virtue of Lemma \ref{lemma-case-2}, we prove {\bf Case 1}. By virtue of Lemma \ref{lemma-case-3}, we prove {\bf Case 2}. In the following, we prove {\bf Case 3.}

		 \vskip 0.5cm 
	{\bf Case 3.} \emph{Suppose that $\mbox{Re} \, x $ and $ \mbox{Im} \, x$ are in the forms 
	\begin{equation*}
		\mbox{Re} \, x=\lambda e+(I-e)\mbox{Re} \, x \quad \mbox{ and } \quad \mbox{Im} \, x=\eta f+(I-f)\mbox{Im} \, x
	\end{equation*}
	}
	\emph{where $e$  (resp. $f$) is an infinite, spectral projection for $\mbox{Re} \, x$ (resp. $\mbox{Im} \, x $) with $I-e$ (resp. $I-f$) finite and nonzero.}
	\vskip 0.5cm

	Without loss of generality, we assume that $\tau(I-e) \ge \tau (I-f)>0$. Since $\M$ is a semifinite, properly infinite factor,  there exists a system of matrix units $\{e_{ij}\}_{i,j\in\NNN}$ for $ \M $ satisfying 
	\begin{enumerate}
		\item \quad $e_{11}= I-e$ and for each $j\in\NNN$, $e_{jj} $ reduces $\mbox{Im} \, x$;
		\item \quad {\scriptsize {SOT}}-$\sum^{}_{i\in\NNN}e_{ii}= I$;
		\item \quad  $e^*_{ji}=e_{ij}$ for all  $i, j\in\NNN$;
		\item \quad $e_{ij}e_{kl}=\delta_{jk}e_{il}$ for all  $i, j, k, l\in\NNN$;
		\item \quad $\{e_{ij}\}_{i,j\in\NNN}\subset \F(\M,\tau)$ and by virtue of Lemma \ref{prelim_lemma1}, there is a uniform upper bound $\eta>0$ such that $\Phi(e_{ij})\leq \eta$ for all $i,j\in\NNN$.
	\end{enumerate} 	
	Define $\P$ to be the von Neumann algebra generated by $\{e_{ij}\}_{i,j\in \NNN}$. It follows that $\P$ is a type ${\rm I}_{\infty}$ subfactor of $\M$, which is $\ast$-isomorphic to $\B(l^2(\NNN))$. Define $\N=\P^{\prime}\cap \M$. Then $\N$ is a finite subfactor of $\M$, which is $\ast$-isomorphic to $e_{11} \M e_{11}$ and $(\N\cup \P)^{\prime\prime}=\M$. It follows that $\M$ is $\ast$-isomorphic to the von Neumann algebra tensor product $\mathcal P  \otimes \mathcal N$, which is denoted by $\mathcal M\cong \mathcal P \otimes \mathcal N$.
	
	By virtue of Claim \ref{normal-case-1}.1, there exist two invertible positive operators $a$ and $b$ in $\N$ such that $a+ib$ is irreducible in $\N$. Define a self-adjoint operator $h$ in $\M$ of the form
	\begin{equation*}
	\begin{aligned}
		h=& \mbox{Im}\, x+\frac{\epsilon \cdot a(e_{12}+e_{21})}{ 3 \cdot \eta \cdot \Vert a \Vert} + \frac{\epsilon  \cdot b}{ \eta\cdot \Vert b \Vert} \sum\nolimits^{\infty}_{j=2} \frac{1}{ 3^{j}} (e_{j,j+1}+e_{j+1,j}).
	\end{aligned}
	\end{equation*}
	As an application of Lemma \ref{prelim_lemma1}, it follows that $\Phi(\mbox{Im}\, x -h) \le \epsilon/2$.
	 
	Define $y=\mbox{Re}\, x+i h$.  Thus $\Phi(x-y)< \epsilon$.  To prove that $y$ is irreducible in $\M$, it is sufficient to show that
	\begin{equation*}
		 \{e_{ij}\}_{i,j \in \NNN}\cup \{a,b\}\subset  W^{\ast}(y).
	\end{equation*}
	
	Since $e_{11}\in W^{\ast}(\mbox{Re}\, x)\subset W^{\ast}(y)$ and $e_{11}\cdot \mbox{Im}\, x=\mbox{Im}\, x \cdot e_{11}$, we have
	\begin{equation*}
		e_{11}h(I-e_{11})= \frac{\epsilon \cdot a\cdot e_{12}}{ 3 \cdot \eta \cdot \Vert a \Vert} \in  W^{\ast}(y).
	\end{equation*} 
	Note that $a$ is positive and invertible in $\N$. The uniqueness of the polar decomposition of $ae_{12}$ entails that $e_{12}$ is in $W^{\ast}(y)$. It follows that $e_{22}\in W^{\ast}(y)$. Since $e_{22}\cdot \mbox{Im}\, x=\mbox{Im}\, x \cdot e_{22}$, we have
	\begin{equation*}
		e_{22}h(I-e_{11}-e_{22})= \frac{\epsilon \cdot b \cdot e_{23}}{ 3^2 \cdot \eta \cdot \Vert b \Vert} \in  W^{\ast}(y).
	\end{equation*} 
	Similarly, we have $\{e_{23},e_{32}, e_{33}\} \subset W^{\ast}(y)$. By induction, it follows that
	$$\{e_{ij}\}_{i,j\in \NNN}  \subset  W^{\ast}(y).$$
	Thus, we have $\P \subseteq W^{\ast}(y)$. This implies that,  for all $j\ge 1$, $ae_{jj} \in W^{\ast}(y)$.
	It follows that 
	\begin{equation*}
	 	a=\mbox{\scriptsize{SOT}-}\sum\nolimits^{\infty}_{j=1}ae_{jj}\in  W^{\ast}(y).
	\end{equation*}
	A similar calculation shows that $b\in  W^{\ast}(y)$.
	Therefore, we have $\P \cup \{a,b\} \subset  W^{\ast}(y)$.
	 
	We are ready to prove that $y$ is irreducible in $\M$. A similar argument as in Lemma \ref{lemma-case-1} entails that  $y=\mbox{Re}\, x+ih$ is irreducible in $\M$. This completes the proof of $\textbf{Case 3}$. 
	 
	To prove  $\textbf{Case 4}$, we assume that $\mbox{Re}\, x$ is a scalar multiple of the identity. We apply a similar argument as in Lemma \ref{lemma-case-3} to construct a self-adjoint, diagonal operator $d$ in the form
	\begin{equation*}
		d=\mbox{Re}\, x+\sum^{\infty}_{j=1}\frac{\epsilon}{\eta} \left(\frac{e_{jj}}{3^{(j+1)}} + \frac{p_{21}e_{jj}p_{12}}{5^{(j+1)}}  + \frac{p_{31}e_{jj}p_{13} }{7^{(j+1)}}\right),
	\end{equation*}
	where $\{p_{ij}\}^{3}_{i,j=1}$ is a system of matrix units for $\M$ such that
	\begin{enumerate}
		\item[(1)] \quad  $p_{11}+p_{22}+p_{33}=I$ for all $j=1,2,3$;
		\item[(2)] \quad  $p^*_{ji}=p_{ij}$ for all  $i, j=1,2,3$;
		\item[(3)] \quad  $p_{ij}p_{kl}=\delta_{jk}p_{il}$ for all  $i, j, k, l=1,2,3$;
		\item[(4)] \quad  $p_{jj}\cdot (\mbox{Im}\, x)=(\mbox{Im}\, x)\cdot  p_{jj} $  for all $j=1,2,3$;
	\end{enumerate}
	and there exists a system of matrix units $\{e_{ij}\}_{i,j\in\NNN}$ for $p_{11} \M p_{11}$ satisfying 
	\begin{enumerate}
		\item[(5)] \quad  {\scriptsize {SOT}}-$\sum^{\infty}_{i=1}e_{ii}=p_{11}$;
		\item[(6)] \quad  $e^*_{ji}=e_{ij}$ for all  $i, j\in\NNN$;
		\item[(7)] \quad  $e_{ij}e_{kl}=\delta_{jk}e_{il}$ for each  $i, j, k, l\in\NNN$;
		\item[(8)] \quad  $\{e_{ij}\}_{i,j\in\NNN}\subset \F(\M,\tau)$ and by virtue of Lemma \ref{prelim_lemma1}, there is a uniform upper bound $\eta>0$ such that $\Phi(e_{ij})\leq \eta$ for all $i,j\in\NNN$.
	\end{enumerate}
	
	Note that $\Phi(d-\mbox{Re}\, x)\le \epsilon /2$. For $d+i\, \mbox{Im}\, x$, we apply Lemma \ref{lemma-case-1} to construct an irreducible operator $y$ in $\M$ such that $\Phi (d+i\, \mbox{Im}\, x-y)<  \epsilon /2$. It follows that
	\begin{equation*}
		\Phi( x -y )\le \Phi(\mbox{Re}\, x-d)+\Phi(d+i\, \mbox{Im}\, x-y)< \epsilon \quad \mbox{ and } \quad x-y\in{\mathcal{K}^{0}_{\Phi}(\mathcal{M},\tau)^{}}.
	\end{equation*} 
	This completes the proof of $\textbf{Case 4}$. Therefore, we complete the whole proof.
\end{proof}

\begin{proof}[Proof of Theorem \ref{normal-to-irreducible}]
	By virtue of Lemma \ref{normal-case-1} and Lemma \ref{normal-case-2}, we can complete the proof.
\end{proof}

\begin{remark}
	Note that Theorem $\ref{normal-to-irreducible}$, Lemma $\ref{lemma-case-1}$, Lemma $\ref{lemma-case-2}$, and Lemma $\ref{lemma-case-3}$ can be all viewed as  evidences to support Problem $\ref{Problem-1.2}$ to have a positive answer.  
\end{remark}


\begin{thebibliography}{99}

\bibitem{Blackadar} Bruce Blackadar. \emph{Operator algebras. Theory of ${\rm C}^{\ast}$-algebras and von Neumann algebras.} Encyclopaedia of Mathematical Sciences, 122. Operator Algebras and Non-commutative Geometry, III. Springer-Verlag, Berlin, 2006.

\bibitem{Connes} Alain Connes. On the cohomology of operator algebras.\emph{J. Functional Analysis}  \textbf{28}, (1978), no. 2, 248--253.

\bibitem{Connes2} Alain Connes, Jacob Feldman, \and Benjamin Weiss. An amenable equivalence relation is generated by a single transformation. \emph{Ergodic Theory Dynamical Systems} \textbf{1}, (1981), no. 4, 431--450.

\bibitem{Connes3} Alain Connes. Classification of injective factors. Cases ${\rm II}_1$, ${\rm II}_{\infty}$, ${\rm III}_{\lambda}$, $\lambda\neq 1$.\emph{Ann. of Math.}  \textbf{2}, 104, (1976), no. 1, 73--115.

\bibitem{Davidson} Kenneth Davidson. \emph{${\rm C}^{\ast}$-algebras by example.} Fields Institute Monographs, 6. American Mathematical Society, Providence, RI, 1996.

\bibitem{Elliott} George Elliott \and  David Evans. The structure of the irrational rotation ${\rm C}^{\ast}$-algebra.\emph{Ann. of Math.}  \textbf{2}, 138, (1993), no. 3, 477--501.

\bibitem{Fang2} Junsheng Fang, Don Hadwin, Eric Nordgren, \and Junhao Shen. Tracial gauge norms on finite von Neumann algebras satisfying the weak Dixmier property. \emph{J. Funct. Anal.} \textbf{255} (2008), no. 1, 142--183.

\bibitem{Fang} Junsheng Fang, Rui Shi, \and Shilin Wen. On irreducible operators in factor von Neumann algebras. \emph{Linear Algebra Appl.} \textbf{565}, (2019), 239--243.

\bibitem{Gohberg} Israel Gohberg \and Mark Kre{$\breve{\imath}$}n. \emph{Introduction to the theory of linear non-selfadjoint operators in Hilbert space.} (Russian), Izdat. ``Nauka'', Moscow, 1965.
	
\bibitem{Hal} Paul Halmos. Irreducible operators. \emph{Michigan Math J.} \textbf{15}, (1968), 215--223.


\bibitem{Kadison1} Richard  Kadison  and  John   Ringrose. \emph{Fundamentals of the theory of operator algebras. Vol. I. Elementary theory.} Reprint of the 1983 original. Graduate Studies in Mathematics, 15. American Mathematical Society, Providence, RI, 1997.

\bibitem{Kadison2} Richard  Kadison  and  John   Ringrose. \emph{Fundamentals of the theory of operator algebras. Vol. II. Advanced theory.} Corrected reprint of the 1986 original. Graduate Studies in Mathematics, 16. American Mathematical Society, Providence, RI, 1997.

\bibitem{Kuroda} Shige Toshi Kuroda. On a theorem of Weyl-von Neumann. \emph{Proc. Japan Acad.} \textbf{34} (1958), 11--15.




\bibitem{Li} Qihui Li, Junhao Shen, Rui Shi. A generalization of the Voiculescu theorem for normal operators in semifinite von Neumann algebras. arXiv:1706.09522 [math.OA].

\bibitem{Murray} Francis Joseph Murray and John Von Neumann. On rings of operators. \emph{Ann. of Math.} (2) \textbf{37} (1936), no. 1, 116--229.


\bibitem{Von2} John von Neumann. Charakterisierung des Spektrums eines Integraloperators. {\em Actualits Sci. Indust.} \textbf{229}, Hermann, Paris, 1935.


\bibitem{Pearcy_2} Carl Pearcy. On certain von Neumann algebras which are generated by partial isometries. \emph{Proc. Amer. Math. Soc.}  \textbf{15}, (1964), 393--395.

\bibitem{Pisier} Gilles Pisier  \and  Quanhua Xu.  Non-commutative $L^{p}$-spaces, {\em Handbook of the geometry of Banach spaces.} North-Holland, Amsterdam,  \textbf{2}, (2003), 1459--1517.

\bibitem{Pop} Sorin Popa. On a problem of R. V. Kadison on maximal abelian $\ast$-subalgebras in factors. \emph{Invent. Math.} \textbf{65}, (1981/82), no. 2, 269--281.

\bibitem{Popa2}  Sorin Popa. Notes on Cartan subalgebras in type ${\rm II}_1$ factors. \emph{Math. Scand.} \textbf{57}, (1985), no. 1, 171--188.
		
\bibitem{Rad} Heydar Radjavi and Peter Rosenthal.  Shorter Notes: The Set of Irreducible Operators is Dense. \emph{Proceedings of the American Mathematical Society.} \textbf{21},  (1969), no. 1, p. 256


\bibitem{Rosenberg} Jonathan Rosenberg. \emph{Amenability of crossed products of ${\rm C}^{\ast}$-algebras.} Comm. Math. Phys. \textbf{57} (1977), no. 2, 187--191.


\bibitem{Schatten} Robert Schatten. \emph{Norm-ideals of completely continuous operators.} Ergebnisse der Mathematik und ihrer Grenzgebiete. N. F., Heft 27 Springer-Verlag, Berlin-G{\" o}ttingen-Heidelberg 1960

\bibitem{Saito_2} Noboru Suzuki \and Teishir$\hat{\rm o}$ Sait$\hat{\rm o}$. On the operators which generate continuous von Neumann algebras.\emph{Tohoku Math. J.} \textbf{15}, no. 2, (1963) 277--280.

\bibitem{Voi} Dan Voiculescu. Some results on norm-ideal perturbations of Hilbert space operators. \emph{J. Operator Theory} \textbf{2} (1979), no. 1, 3--37.



\bibitem{Weyl} Hermann Weyl. {\"{U}ber beschr\"{a}nkte quadratische formen, deren differenz vollstetig ist.} \emph{Rend. Circ. Mat. Palermo} \textbf{27} (1) (1909), 373--392.

\bibitem{Wogen} Warren Wogen. On generators for von Neumann algebras. \emph{Bull. Amer. Math. Soc.} \textbf{75}, (1969), 95--99.




\end{thebibliography}
\end{document}